\documentclass[11pt,reqno]{amsart}
\usepackage{latexsym,amssymb,amsfonts,mathrsfs}
\usepackage{amsthm}
\usepackage{amsmath}
\usepackage[a4paper,bottom=2.5cm]{geometry}
\usepackage{a4wide}
\usepackage{enumerate}
\usepackage{comment}
\usepackage[latin1, utf8]{inputenc}
\usepackage{eufrak}
\usepackage[usenames]{color}
\usepackage{graphicx}
\usepackage{lmodern,xfrac,xcolor}
\usepackage[noadjust]{cite}

\usepackage{mathtools,enumerate}
\usepackage[toc]{appendix}
\mathtoolsset{showonlyrefs}
\usepackage{scalerel}
\usepackage{relsize}

\usepackage{tikz}
\usetikzlibrary{positioning}

\usepackage{mathtools,accents}

\usepackage[unicode,plainpages=false]{hyperref}
\mathtoolsset{showonlyrefs}

\newtheorem{theo}{Theorem}
\newtheorem{lem}[theo]{Lemma}
\newtheorem{prop}[theo]{Proposition}
\newtheorem{cor}[theo]{Corollary}

\theoremstyle{definition}
\newtheorem{defi}[theo]{Definition}

\newtheorem{assumptions}{Assumptions}

\newtheorem{exs}[theo]{Examples}

\theoremstyle{remark}
\newtheorem{rem}[theo]{Remark}

\newtheorem{notat}[theo]{Notation}

\newcommand{\ve}{\varepsilon}

\newcommand{\dx}{\,\mathrm{d}x}

\newcommand{\mi}{\mathrm{i}}
\newcommand{\la}{\lambda}
\newcommand{\me}{\mathsf{e}}

\newcommand{\mv}{\mathsf{v}}
\newcommand{\mw}{\mathsf{w}}
\newcommand{\mE}{\mathsf{E}}
\newcommand{\mV}{\mathsf{V}}
\newcommand{\mB}{\mathsf{B}}
\newcommand{\mG}{\mathsf{G}}

\newcommand{\ea}{\mathfrak{a}}

\newcommand{\Id}{\mathrm{Id}}

\newcommand{\comp}{\mathbb{C}}
\newcommand{\real}{\mathbb{R}}
\newcommand{\nat}{\mathbb{N}}
\newcommand{\bbz}{\mathbb{Z}}

\newcommand{\bbP}{\mathbb{P}}

\newcommand{\myhat}[1]{%
  \mathchoice
    {\widehat{#1}}
    {\widehat{#1}}
    {\text{\scriptsize$\widehat{#1}$}}
    {\hat{#1}}
}

\newcommand{\mcB}{\mathcal{B}}
\newcommand{\mcD}{\mathcal{D}}

\newcommand{\mcH}{\mathcal{H}}

\newcommand{\mcM}{\mathcal{M}}
\newcommand{\mcP}{\mathcal{P}}
\newcommand{\mcL}{\mathcal{L}}
\newcommand{\mcQ}{\mathcal{Q}}
\newcommand{\mrQ}{\mathrm{Q}}

\newcommand{\mbQs}{\mathbf{Q_{*}}}
\newcommand{\mbKs}{\mathbf{K_{*}}^{D,m,M}}
\newcommand{\wmrQ}{\myhat{\mrQ}}
\newcommand{\wmcQ}{\widehat{\mcQ}}
\newcommand{\wmG}{\widehat{\mG}}
\newcommand{\wG}{\widehat{G}}
\newcommand{\wL}{\hat{L}}
\newcommand{\wW}{\hat{W}}
\newcommand{\wbeta}{\hat{\beta}}
\newcommand{\wU}{\hat{U}}
\newcommand{\wx}{\hat{x}_0}
\newcommand{\wmv}{\hat{\mv}}
\newcommand{\wmV}{\hat{\mV}}
\newcommand{\wmE}{\hat{\mE}}

\newcommand{\wz}{\hat{z}}

\DeclareMathOperator{\Lip}{Lip}

\numberwithin{equation}{section} \numberwithin{theo}{section}

\begin{document}

\title{Benjamini--Schramm limit of the heat semigroup on quantum graphs}
\author{Mihály Kovács}
\address{Faculty of Information Technology and Bionics\\P\'azm\'any P\'eter Catholic University,
		Budapest, Hungary and
		Department of Differential Equations, Faculty of Natural Sciences, Budapest University of Technology and Economics,
		Budapest, Hungary and
		Chalmers University of Technology and University of Gothenburg,
		Gothenburg, Sweden}
	\email{mkovacs@math.bme.hu}
	
	\author{Eszter Sikolya}
	\address{Department of Applied Analysis and Computational Mathematics\\
		E\"otv\"os Lor\'and University\\
		Budapest, Hungary\\
		Alfr\'ed R\'enyi Institute of Mathematics,\\ 
		Budapest, Hungary}
	\email{eszter.sikolya@ttk.elte.hu}
\date{July 2026}

\begin{abstract}
We study the behaviour of heat semigroups on quantum graphs under Benjamini--Schramm convergence. For quantum graphs with uniformly bounded geometry, equipped with continuity and Kirchhoff vertex conditions, we show that the heat semigroup, transported to a common Hilbert space by a canonical breadth-first identification of the edges, depends continuously on the underlying rooted quantum graph with respect to a local Benjamini--Schramm-type metric. As a consequence, root-averaged pairings of the semigroup converge along Benjamini--Schramm convergent sequences of finite quantum graphs. Combining this with a Trotter--Kato-type approximation of the semigroup by semigroups on metric balls, we obtain a double-limit theorem interchanging the truncation radius and the graph limit.
\end{abstract}


\maketitle
\pagestyle{plain}

\section{Introduction}

In the paper \cite{BS01} Benjamini and Schramm introduced limit objects of sparse graphs, that is, graph sequences having uniformly bounded degrees. Based on this concept, an extensive literature has arisen, see e.g.~\cite{ATV2011} and the book of Lovász \cite{LovaszBook} and references therein. 

Parallel to this, the theory of quantum graphs, that is, metric graphs supplied with second order Hamiltonian operators, developed rapidly. We only refer here to the monographs \cite{BeKu} and \cite{Mugnolobook} and the survey paper \cite{Be17}. 

In 2021, Anantharaman, Ingremeau, Sabri and Winn in \cite{AISW2021} introduced Benjamini--Schramm limit of quantum graphs, and proved the convergence of the integrals of the Hamiltonian's Green's function along such sequences.  In \cite{AISW2021JMPA} the same authors considered quantum graphs converging in the sense of Benjamini--Schramm to the random rooted infinite quantum tree, and showed that under appropriate assumptions, the eigenfunctions of the Hamiltonians with eigenvalues lying in an interval are spatially delocalized.  However, we are not aware of any further remarkable result concerning Benjamini--Schramm limit of quantum graphs.

Kovács and Sikolya together with their coauthors investigated quantum graphs based on operator semigroup techniques, see e.g.~\cite{KMS07}, and quantum graphs with noise, cf.~\cite{KS21JEE}, \cite{KS23EJP}. The question naturally arises how semigroups associated with quantum graph sequences behave when passing to the Benjamini--Schramm limit. Limits of semigroups form a well-researched area, see e.g.~\cite{ItoKappel98} and \cite[Sec.~III.4]{EN00}, since they often appear in applications, e.g.~when dynamical systems are approximated by simpler (e.g.~finite-dimensional) systems. In this paper we successfully combine the concepts of semigroup sequences and quantum graph limits.\\

The paper is organized as follows. In Section \ref{sec:heatsgronquantumgraphs} we consider quantum graphs carrying second order operators (Hamiltonians) equipped with general boundary conditions in the vertices described by unitary matrices. We also introduce the corresponding Hilbert spaces of the edge functions and the strongly continuous semigroup associated with a quantum graph in Proposition \ref{prop:AQgen}.

In Section \ref{sec:approxheatsgr} we first define rooted quantum graphs and introduce the restriction of the associated semigroup to the $r$-ball around the root. Using a Trotter--Kato type approximation result of Ito and Kappel from \cite{ItoKappel98}, in Theorem \ref{thm:heatsgrconv} we show that the restricted semigroups on the $r$-balls converge strongly to the semigroup on the quantum graph when $r$ tends to infinity.

Section \ref{sec:BSlimitquantumgraphs} contains our main results. Based on \cite{AISW2021}, we introduce a metric on the equivalence classes of rooted quantum graphs in Definition \ref{defi:metricrooted}. In Theorem \ref{prop:Ffgtrcontbdd} we show that the function $F$ defined as the weak form of the semigroup associated with an (equivalence class of) rooted quantum graph is bounded and continuous with respect to this metric on a subset of  (equivalence classes of) rooted quantum graphs equipped with standard Kirchhoff-boundary conditions and  uniformly bounded data. Using \cite[Cor.~3.7]{AISW2021} we are then able to claim in Corollary \ref{cor:FBSsubseq} that for any sequence of rooted quantum graphs $(\mrQ_N)$ from this subset, there is a subsequence such that the integrals of $F$ as choosing the root uniformly on $\mrQ_N$ converge to the integral of $F$ with respect to the measure being the Benjamini--Schramm limit of the sequence.

In Theorem \ref{thm:QNdoublelimit} and Corollary \ref{cor:QNkdoublelimit} we combine the result on convergence of the semigroups on balls with the result on Benjamini--Schramm limit of semigroups on quantum graphs.

In Section \ref{sec:examplesBSlimit} we list the most common examples for graph sequences to which our theory applies.

\section{Heat semigroup on quantum graphs}\label{sec:heatsgronquantumgraphs}

\subsection{Quantum graphs}

A quantum graph $\mcQ$ is a metric graph endowed with a second order differential operator. For the next terminology, we follow \cite[Sec.~2]{AISW2021}.

Let  $G=G(\mV,\mE)$ be a connected \emph{combinatorial} or \emph{discrete graph} consisting of a  countable set of vertices $\mV = \{\mv\}$ and edges $\mE = \{\me\}$ connecting (some of) the vertices. We also assume that for each vertex $\mv$, its degree $d(\mv)$ is finite and positive, and $G$ contains no loops (edges from a vertex to itself). 

Denote by $\mcB = \mcB(\mG)$ the set of oriented edges (or bonds), so that  $|\mcB| = 2|\mE|$ (an edge $\me=\{\mv_1,\mv_2 \}\in \mE$ gives rise to two oriented edges 
$(\mv_1, \mv_2), (\mv_2, \mv_1) \in \mcB$). If $b \in \mcB$, we  denote by $\hat{b}$ the reverse bond. We write $o(b)$ for the origin of $b$ and $t(b)$ for the 
terminus of $b$, and  $\me(b) \in \mE$ denotes the edge obtained by forgetting  the orientation of $b$. 

We define the \emph{length function} as a map $L:\mE\to (0,+\infty)$ and denote $L(\me)=L_{\me}$. For $b\in\mcB$, we denote $L_b := L(\me(b))$. 

The \emph{metric graph} is then defined as the set of equivalence classes
\begin{equation}
    \mG := \{ (b, x) ; b \in \mcB, x \in (0, L_b) \} / \simeq \ ,
\end{equation}
where the equivalence relation is given by
\begin{equation}
    (b, x) \simeq (b', x')\text{ if }b' = \hat{b}\text{ and }x' = L_b - x.
\end{equation}
Note that the points of $\mG$ correspond to the interior points of the edges; in particular, the vertices of $G$ are not points of $\mG$.
A function $z : \mG \to \comp$ can be thus written as $z = (z_b)_{b \in \mcB}$, where $z_b(x) := z(b, x)$ satisfying the condition $z_{\hat{b}}(L_b - x) = z_b(x)$.
If $z$ is measurable, we define
\begin{equation}\label{eq:integralbound}
    \int_{\mG} z(\mathbf{x}) \, \mathrm{d}\mathbf{x} := \frac{1}{2} \sum_{b \in \mcB} \int_0^{L_b} z(b, x) \, \mathrm{d}x \ ,
\end{equation}
where the points of $\mG$ are denoted by $\mathbf{x}$. (The $\frac{1}{2}$ factor comes from the fact that each non-oriented edge is counted twice.)

We refer here more precisely to the definition \cite[Def.~2.2]{AISW2021}.

\begin{defi}\label{defi:quantumgraph}
A \emph{quantum graph} $\mcQ =(\mG, W, \beta, U)$ is the data of:
\begin{itemize}
	\item A metric graph $\mG=(\mV,\mE,L)$.
\item A potential $W=(W_{b})_{b\in\mcB} \in \bigoplus_{b\in\mcB}C([0,L_b])$, $W_{b}\geq 0$ satisfying $W_{\hat{b}}(L_b - x) = W_b(x)$, $b\in\mcB$.
\item For each $\mv \in \mV$, a labeling of the oriented edges starting at $\mv$, that is, a bijective map
\begin{equation}\label{eq:betav}
    \beta^{\mv} : \{1, \dots, d(\mv)\} \longrightarrow \{b \in \mcB : o(b) = \mv\}.
\end{equation}
\item For each $\mv\in \mV$, a unitary matrix $U_{\mv}\in \mcM_{d(\mv)}(\comp)$ describing the vertex conditions in $\mv$ for the domain of the operator associated with $\mcQ$, where $d(\mv)$ denotes the degree of the vertex $\mv$.
\end{itemize}
\end{defi}

\begin{rem}\label{rem:metricgraphbonds}
In this paper, to make notations simpler, our terminology follows \cite[Chap.~1]{BeKu} (see also \cite{Be17} and \cite{Mugnolobook}). Thus, we simply consider edges $\me\in\mE$ and functions defined on them as 
\[z=\left(z_{\me}\right)_{\me\in \mE}.\]
Hence, for a measurable function $z$ on $\mG$, instead of \eqref{eq:integralbound} we will write
\begin{equation}
    \int_{\mG}z(x)\dx \coloneqq \sum_{\me\in\mE}\int_0^{L_{\me}}z_{\me}(x)\dx_{\me},
\end{equation}
where $\dx_{\me}$ denotes the Lebesgue-measure on the edge $\me$.

However, we always keep in mind and make use of the precise structure of the metric graph $\mG$ including bonds such that e.g.~we can consider arbitrary oriented parametrization of the edges. Where necessary, we clarify the concepts in accordance with this.

In Sections \ref{sec:heatsgronquantumgraphs} and \ref{sec:approxheatsgr} we allow complex-valued functions and complex unitary vertex matrices; all function spaces are understood over $\comp$, and all scalar products are sesquilinear (linear in the first and conjugate-linear in the second argument). In Section \ref{sec:BSlimitquantumgraphs} we will restrict our attention to quantum graphs with continuity and Kirchhoff vertex conditions, for which the vertex matrices are real; from that point on we work with real-valued functions and real Hilbert spaces, see the discussion preceding the definition of the class $\mbKs$ there.
\end{rem}

In the sequel, we will always make the following assumption
\begin{assumptions}\label{ass:main}
\begin{equation}\label{eq:graphparameterbdd}
    \begin{split}
    d(G) &\coloneqq \sup_{\mv\in \mV}d(\mv) <\infty,\\
\underline{L}(\mcQ) &\coloneqq\inf_{\me\in \mE}L_{\me} > 0,\\
\overline{L}(\mcQ) &\coloneqq\sup_{\me\in \mE}L_{\me} <\infty,\\
\|W\|_{\infty} &\coloneqq \sup_{{\me}\in\mE}\|W_{\me}\|_{\infty} <\infty,\\
S&\coloneqq \sup_{\mv\in\mV}\|\Lambda_{\mv}\|<\infty,
\end{split}
\end{equation}
where  $\Lambda_{\mv}$ is the Cayley-transform
\begin{equation}\label{eq:LambdavfromUv}
    \Lambda_{\mv}=-\mi\left(U_{\mv}+\Id\right)_R^{-1}\left(U_{\mv}-\Id\right)_R,
\end{equation}
where $\left(U_{\mv}\pm\Id\right)_R$ denotes the restriction of $U_{\mv}\pm\Id$ to the direct sum of the eigenspaces corresponding to the eigenvalues of $U_{\mv}$ other than $\pm1$.

\end{assumptions}

\begin{defi}
A quantum graph $\mcQ =(\mG, W,\beta, U)$ will be called \emph{finite} if $|\mV|, |\mE| <\infty $. Its \emph{total length} is then defined as
\[\mcL(\mcQ) =\sum_{\me\in \mE}L_{\me}.\]
\end{defi}

For a quantum graph $\mcQ =(\mG, W,\beta, U)$ we define the following spaces. Denote by $H^k(\me)$ the Sobolev space of functions on the segment $\me$ whose distributional derivatives up to order $k$ belong to $L^2(\me)$.

\begin{defi}\label{defi:mcHmch2}
We consider the Hilbert spaces
\begin{enumerate}
	\item 
	\begin{equation}\label{eq:mcH}
		\mcH_{\mcQ}= \left\{f=(f_{\me})_{\me\in\mE}\in\bigoplus_{\me\in\mE}L^2(\me): \|f\|^2_{L^2(\mG)}\coloneqq\sum_{\me\in\mE}\|f_{\me}\|^2_{L^2(\me)}<\infty\right\},
	\end{equation}
    
\item	\begin{equation}\label{eq:mcH2}
\mcH^2_{\mcQ}=\left\{f=(f_{\me})_{\me\in\mE}\in\bigoplus_{\me\in\mE}H^2(\me): \|f\|^2_{H^2(\mG)}\coloneqq\sum_{\me\in\mE}\|f_{\me}\|^2_{H^2(\me)}<\infty\right\},\end{equation}
	\end{enumerate}
The	finiteness conditions for the norms are needed for infinite graphs only. 
\end{defi}	

To write down the vertex conditions for the operator associated with the quantum graph, for a given function $z\in\mcH^2_{\mcQ}$, and for each $\mv\in\mV$, we introduce the following notation.

We set
\begin{equation}\label{eq:FvFvvbond}
Z(\mv) := \left( z(o_{\beta^{\mv}(j)}) \right)_{j=1}^{d(\mv)} , \qquad Z'(\mv) := \left( z'(o_{\beta^{\mv}(j)}) \right)_{j=1}^{d(\mv)}\in\comp^{d(\mv)},
\end{equation}
where, in the bond notation of Remark \ref{rem:metricgraphbonds}, $z(o_{b}) := z_{b}(0)$ and $z'(o_{b}) := z_{b}'(0)$ denote the value and the derivative of $z$ at the origin of the bond $b$. In the definition of $Z'(\mv)$, this corresponds to assuming that derivatives are taken in the directions away from the vertex $\mv$, cf.~\cite[Sec.~1.4.]{BeKu}.

For the unitary matrices $U_{\mv}$, $\mv\in\mV$ from Definition \ref{defi:quantumgraph} define
\begin{equation}\label{eq:A1vA2v}
A_{\mv}\coloneqq \mi\left(U_{\mv}-\Id\right),\quad B_{\mv}\coloneqq U_{\mv}+\Id, 
\end{equation}
and for $z\in\mcH^2_{\mcQ}$, we set the vertex conditions
\begin{equation}\label{eq:vertexcond}
A_{\mv}Z(\mv)+B_{\mv}Z'(\mv)=0\text{ for all }\mv\in\mV.
\end{equation}

It is straightforward, see \cite[Examp.~2.3.(a)]{AISW2021}, that letting
\begin{equation}\label{eq:UvconKN}
U_{\mv}=\frac{2}{d(\mv)}\cdot\mathbf{1}-\Id,
\end{equation}
where $\mathbf{1}$ is the matrix with all entries equal to $1$, condition \eqref{eq:vertexcond} turns to the following ones, where  $\mE_{\mv}$ denotes the set of edges incident to the vertex $\mv$:
\begin{enumerate}
	\item \textbf{continuity} in the vertices, that is,
	\[\text{for all }\mv\in\mV\text{ and }\me,\me'\in\mE_{\mv}: z_{\me}(\mv)=z_{\me'}(\mv),\]

and
\item \[\text{for all }\mv\in \mV: \sum_{\me\in\mE_{\mv}}z_{\me}'(\mv)=0\]
which is called \textbf{Kirchhoff's} or \textbf{Kirchhoff--Neumann-condition} being a generalization of the usual Neumann-condition. This condition turns to the one-dimensional Neumann-condition in vertices of degree one.
\end{enumerate}

We are now ready to define the \emph{Hamiltonian} operator of the quantum graph $\mcQ =(\mG, W,\beta, U)$ as follows.

\begin{defi}
Let $\left(A_{\mcQ},\mcD(A_{\mcQ})\right)$ be the operator defined on  $\mcH_{\mcQ}^2$ as
\begin{equation}\label{eq:opAQoriginal}
\begin{split}
(A_{\mcQ} z)_{\me}&= z_{\me}''-W_{\me}z_{\me},\quad \me\in\mE,\\
\mcD(A_{\mcQ})&=\left\{z\in\mcH^2_{\mcQ}:A_{\mv}Z(\mv)+B_{\mv}  Z'(\mv)=0,\, \mv\in\mV\right\}.
\end{split}
\end{equation}
\end{defi}

For technical reasons, in the following we will always rescale the edges of the quantum graph and the corresponding function spaces. 

\begin{defi}\label{defi:mcHQ}
    For a function $z$ on the metric graph $\mG=(\mV,\mE,L)$, we define its rescaling as
    \begin{equation}\label{eq:fctnrescaled}
        \tilde{z}_{\me}(x)\coloneqq z_{\me}(x\cdot L_{\me}),\quad x\in [0,1],\; \me\in\mE.
    \end{equation}
We define the Hilbert space of the rescaled functions as
    \begin{equation}\label{eq:mcHtilde}
		\widetilde{\mcH}_{\mcQ}= \left\{\tilde{z}=(\tilde{z}_{\me})_{\me\in\mE}\in\bigoplus_{\me\in\mE}L^2_{\mcQ}(\me,w): \|\tilde{z}\|_{\widetilde{\mcH}_{\mcQ}}^2\coloneqq \sum_{\me\in\mE}\|\tilde{z}_{\me}\|^2_{L^2_{\mcQ}(\me,w)}<\infty\right\},
	\end{equation}
    where $\tilde{z}_{\me}\in L^2_{\mcQ}(\me,w)$ if and only if $\tilde{z}_{\me}\in L^2(\me)$ and
    \begin{equation}        
    \|\tilde{z}_{\me}\|^2_{L^2_{\mcQ}(\me,w)}=\int_0^1|\tilde{z}_{\me}(x)|^2 L_{\me}\dx.
    \end{equation}
    We also introduce the scalar product obtaining the norm on $\widetilde{\mcH}_{\mcQ}$ as
\begin{equation}        
    \langle\tilde{z},\tilde{u}\rangle_{\widetilde{\mcH}_{\mcQ}}=\sum_{\me\in\mE}\int_0^1\tilde{z}_{\me}(x)\overline{\tilde{u}_{\me}(x)} L_{\me}\dx,\quad \tilde{z},\tilde{u}\in\widetilde{\mcH}_{\mcQ}.
    \end{equation}    
    In this way, if we denote the space from \eqref{eq:mcH} by $\mcH_{\mcQ}$,  the mapping
    \begin{equation}\label{eq:Psiisometry}
       \Phi\colon\mcH_{\mcQ}\to\widetilde{\mcH}_{\mcQ}, \quad   \Phi z\coloneqq \tilde{z}
    \end{equation}
    is a unitary isomorphism between $\mcH_{\mcQ}$ and $\widetilde{\mcH}_{\mcQ}$ according to Assumption \ref{ass:main} and \cite[p.~67]{Mugnolobook}.
\end{defi}

We also have to define the rescaled form of the Hamiltonian operator \eqref{eq:opAQoriginal} on $\widetilde{\mcH}_{\mcQ}$ as follows.

\begin{defi}\label{defi:opAQ}
Let $\left(\widetilde{A}_{\mcQ},\mcD(\widetilde{A}_{\mcQ})\right)$ be the operator defined on  $\widetilde{\mcH}_{\mcQ}$ as
\begin{align}
(\widetilde{A}_{\mcQ} \tilde{z})_{\me}&=\frac{1}{L_{\me}^2} \tilde{z}_{\me}''-\widetilde{W}_{\me}\tilde{z}_{\me},\quad \me\in\mE,\label{eq:opAQ}\\
\mcD(\widetilde{A}_{\mcQ})&=\left\{\tilde{z}\in\widetilde{\mcH}_{\mcQ}^2:A_{\mv}\tilde{Z}(\mv)+B_{\mv} \cdot \mathrm{diag}\left(\frac{1}{L_{\me}}\right)_{\me\in\mE_{\mv}} \cdot\tilde{Z}'(\mv)=0,\, \mv\in\mV\right\}\label{eq:domopAQ}
\end{align}
and $\tilde{Z}(\mv)$, $\tilde{Z}'(\mv)$ are defined in \eqref{eq:FvFvvbond} for the rescaled functions $\tilde{z}$, and 
\begin{equation}\label{eq:mcH2tilde}
		\widetilde{\mcH}^2_{\mcQ}= \left\{\tilde{z}=(\tilde{z}_{\me})_{\me\in\mE}\in\bigoplus_{\me\in\mE}H^2_{\mcQ}(\me,w): \sum_{\me\in\mE}\|\tilde{z}_{\me}\|^2_{H_{\mcQ}^2(\me,w)}<\infty\right\},
	\end{equation}
    with $\tilde{z}_{\me}\in H_{\mcQ}^2(\me,w)$ holds if and only if $\tilde{z}_{\me}\in H_{\mcQ}^2(\me)$ and
    \begin{equation}        
    \|\tilde{z}_{\me}\|^2_{H^2_{\mcQ}(\me,w)}=\|\tilde{z}_{\me}\|^2_{L^2_{\mcQ}(\me,w)}+\|\tilde{z}'_{\me}\|^2_{L^2_{\mcQ}(\me,w)}+\|\tilde{z}''_{\me}\|^2_{L^2_{\mcQ}(\me,w)},
    \end{equation}
 cf.~\eqref{eq:mcH2} and \eqref{eq:mcHtilde}.
\end{defi}

\textbf{Convention.} From now on, for a quantum graph $\mcQ =(\mG, W,\beta, U)$ we always consider the functions $z$ and $W$ on $\mG$ in their rescaled form such that the coordinates of $z_{\me}$ and  $W_{\me}$, $\me\in\mE$, are defined on $[0,1]$, see \eqref{eq:fctnrescaled}.  Accordingly, we also consider the Hamiltonian operator as in Definition \ref{defi:opAQ} on the space $\widetilde{\mcH}_{\mcQ}$. To simplify notations, hereinafter we omit the tilde sign.\medskip

\begin{rem}
By \cite[Thm.~1.4.4.]{BeKu}, the following three definitions of $\mcD(A_{\mcQ})$ are equivalent:
\begin{enumerate}
    \item The domain $\mcD(A_{\mcQ})$ is defined as in \eqref{eq:domopAQ} with matrices $A_{\mv}$ and $B_{\mv}$ satisfying the following two conditions:
\begin{equation}
    \begin{split}
        &\mathrm{rank}[A_{\mv}\, B_{\mv}]\text{ is maximal},\\
        &A_{\mv} B_{\mv}^*\text{ is self-adjoint}
    \end{split}
\end{equation}
for each $\mv\in\mV$;
\item For every $\mv\in\mV$ there exists a $d(\mv)\times d(\mv)$ unitary matrix $U_{\mv}$ such that the boundary conditions in \eqref{eq:domopAQ} can be written as
\begin{equation}
    \mi \left(U_{\mv}-\Id\right)Z(\mv)+\left(U_{\mv}+\Id\right)\cdot \mathrm{diag}\left(\frac{1}{L_{\me}}\right)_{\me\in\mE_{\mv}} \cdot Z'(\mv)=0;
\end{equation}
\item For every $\mv\in\mV$ there are three orthogonal
(and mutually orthogonal) projections $P_{D,\mv}$, $P_{N,\mv}$ and $P_{R,\mv}=\Id-P_{D,\mv}-P_{N,\mv}$ acting on $\comp^{d(\mv)}$ and an invertible self-adjoint operator $\Lambda_{\mv}$ -- the Cayley-transform of $U_{\mv}$ defined in \eqref{eq:LambdavfromUv} -- acting in the
subspace $P_{R,\mv}\comp^{d(\mv)}$, such that the boundary values of $z$ satisfy
\begin{equation}\label{eq:domAQprojections}
    \begin{cases}
        P_{D,\mv}Z(\mv)=0, & \text{"Dirichlet part"},\\
        P_{N,\mv}\cdot \mathrm{diag}\left(\frac{1}{L_{\me}}\right)_{\me\in\mE_{\mv}} \cdot Z'(\mv)=0, & \text{"Neumann part"},\\
        P_{R,\mv}\cdot \mathrm{diag}\left(\frac{1}{L_{\me}}\right)_{\me\in\mE_{\mv}} \cdot Z'(\mv)=\Lambda_{\mv}P_{R,\mv}Z(\mv), & \text{"Robin part"}.
    \end{cases}
\end{equation}
\end{enumerate}
By \cite[Rem.~1.4.5., Lem.~1.4.7.]{BeKu}, we also know that in this case the unitary matrices
\begin{equation}\label{eq:UvfromAB}
    U_{\mv}=-\left(A_{\mv}-\mi B_{\mv}\right)^{-1}\left(A_{\mv}+\mi B_{\mv}\right)
\end{equation}
are defined uniquely, however, the matrices $A_{\mv}$ and $B_{\mv}$ are not.
\end{rem}

\begin{rem}\label{rem:contwKNcond}
For the rescaled functions, the continuity and Kirchhoff's condition reads as
\begin{enumerate}
	\item \textbf{continuity} in the vertices, that is,
	\begin{equation}\label{eq:wcontcond}
	    \text{for all }\mv\in\mV\text{ and }\me,\me'\in\mE_{\mv}: z_{\me}(\mv)=z_{\me'}(\mv),
	\end{equation}
hence, $z\in C(\mG)$ holds; and
\item the \textbf{weighted Kirchhoff's condition} as
\begin{equation}\label{eq:wKNcond}
    \text{for all }\mv\in \mV: \sum_{\me\in\mE_{\mv}}\frac{1}{L_{\me}}z_{\me}'(\mv)=0
\end{equation}
being a generalization of the usual Neumann-condition.
\end{enumerate}
Furthermore, it is straightforward (see \cite[Ex.~1.4.4.]{BeKu}) that for all $\mv\in\mV$, these conditions can be written in terms of projections from \eqref{eq:domAQprojections} with
\begin{equation}\label{eq:KNprojections}
\begin{split}
   P_{D,\mv}&\text{ is the orthogonal projection onto the kernel of }B_{\mv},\\
    P_{N,\mv}&=\Id-P_{D,\mv},\\
    P_{R,\mv}&=0.
\end{split}
\end{equation}
\end{rem}

\subsection{Heat semigroup on quantum graphs}

For a quantum graph $\mcQ =(\mG, W,\beta, U)$ and the operator \eqref{eq:opAQ}, \eqref{eq:domopAQ}, we consider the following abstract Cauchy problem on $\mcH_{\mcQ}$:
\begin{equation}\label{eq:diffACP}
\begin{split}
\dot{z}(t)&=A_{\mcQ}z(t),\\
z(0)&=z_0.
\end{split}
\end{equation}
In the same way as in \cite[Thm.~1.4.19.]{BeKu}  one can prove the following.
\begin{prop}
The operator $\left(-A_{\mcQ},\mcD(A_{\mcQ})\right)$  defined in \eqref{eq:opAQ}--\eqref{eq:domopAQ}, is the operator associated with the form
\begin{equation}\label{eq:formAQ}
\begin{split}
\ea_{\mcQ}(z,h)&=\sum_{\me\in\mE}\int_0^{1}\left(\frac{1}{L_{\me}^2}z_{\me}'(x)\overline{h_{\me}'(x)}+W_{\me}(x)z_{\me}(x)\overline{h_{\me}(x)}\right)L_{\me}\dx\\
&+\sum_{\mv\in\mV}\langle\Lambda_{\mv}P_{R,\mv}Z(\mv),P_{R,\mv}H(\mv)\rangle ,\\
\mcD(\ea_{\mcQ})&=\left\{z\in \bigoplus_{\me\in\mE} H^1(\me)\colon P_{D,\mv}Z(\mv)=0\text{ for all }\mv\in\mV\right\},
\end{split}
\end{equation}
where $P_{D,\mv}$, $P_{R,\mv}$ and $\Lambda_{\mv}$ are defined in \eqref{eq:domAQprojections}, and $\langle \cdot,\cdot\rangle$ denotes the usual $\ell^2$-inner product of the vectors.
This means that 
\begin{equation}
\begin{split}
\mcD(A_{\mcQ})&=\left\{z\in \mcD(\ea_{\mcQ}) \colon \exists\, g\in \mcH_{\mcQ}\text{ s.t. }\ea_{\mcQ}(z,h)=\langle g,h\rangle_{\mcH_{\mcQ}}\text{ for all }h\in \mcD(\ea_{\mcQ})\right\},\\
-A_{\mcQ} z&=g.
\end{split}
\end{equation}
\end{prop}

The properties of the above form imply nice properties of the operator $(A_{\mcQ},\mcD(A_{\mcQ}))$  and the mild solutions of the  abstract Cauchy problem \eqref{eq:diffACP} given by the orbits of the strongly continuous semigroup generated by $A_{\mcQ}$. For the terminology we refer to \cite[Sec.~II.6.]{EN00} and \cite[Sec.~3]{ABHN11}.

\begin{prop}\label{prop:AQgen}
The operator $(A_{\mcQ},\mcD(A_{\mcQ}))$ is densely defined, self-adjoint and bounded from above. 
Hence, the strongly continuous semigroup $(S_{\mcQ}(t))_{t\geq 0}$ generated by $(A_{\mcQ},\mcD(A_{\mcQ}))$ is quasicontractive and analytic.

As a consequence, the abstract Cauchy problem \eqref{eq:diffACP} is well-posed and its unique mild solutions are given by the orbits of the semigroup $S_{\mcQ}(t)z_0,$ $t\geq 0$ for each $z_0\in\mcH_{\mcQ}.$
\end{prop}

\begin{proof}
All the properties of $A_{\mcQ}$, except for the boundedness, can be seen in the same way as a consequence of the properties of $\ea_{\mcQ}$, see \cite[Prop.~2.3.]{KS23EJP}.  To prove that the self-adjoint operator $(A_{\mcQ},\mcD(A_{\mcQ}))$ is bounded from above, by a straightforward computation
\begin{equation}
    \begin{split}
      \langle A_{\mcQ}z,z \rangle_{\mcH_{\mcQ}}&=-\ea_{\mcQ}(z,z)=-\sum_{\me\in\mE}\int_0^{1}\left(\frac{1}{L_{\me}^2}|z_{\me}'(x)|^2+W_{\me}(x)|z_{\me}(x)|^2\right)L_{\me}\dx\\
&-\sum_{\mv\in\mV}\langle\Lambda_{\mv}P_{R,\mv}Z(\mv),P_{R,\mv}Z(\mv)\rangle 
\end{split}
\end{equation}
By Assumption \ref{ass:main}, 
observe that
\begin{equation}
      \langle A_{\mcQ}z,z \rangle_{\mcH_{\mcQ}}\leq -\sum_{\me\in\mE}\frac{1}{L_{\me}^2}\|z_{\me}'\|^2_{L^2_{\mcQ}(\me,w)}+S\sum_{\mv\in\mV}\|Z(\mv)\|^2.
      \end{equation}
    
By a parametrized version of the trace estimate behind \cite[(12)]{BolinKovacs2024} -- which follows, for $u\in H^1(0,\ell)$ and any $a>0$, from $|u(x)|^2\leq \frac{1}{\ell}\|u\|^2_{L^2(0,\ell)}+2\|u\|_{L^2(0,\ell)}\|u'\|_{L^2(0,\ell)}$ and Young's inequality, cf.~also the proof of \cite[Thm.~1.4.19]{BeKu} -- and using the isometry in \eqref{eq:Psiisometry}, we have
\begin{equation}\sum_{\mv\in\mV}\|Z(\mv)\|^2\leq 2\sum_{\me\in\mE}\left(\left(1+\frac{1}{a}\right)\frac{1}{L_ {\me}}\|z_{\me}\|^2_{L^2_{\mcQ}(\me,w)}+\frac{a}{L_{\me}}\|z_{\me}'\|^2_{L^2_{\mcQ}(\me,w)}\right)\end{equation}
with any parameter satisfying
\begin{equation}
    0<a<\frac{1}{2S\cdot \overline{L}(\mcQ)}
\end{equation}
(if $S=0$, any $a>0$ is admissible).
Thus,
\begin{equation}
    \begin{split}
      \langle A_{\mcQ}z,z \rangle_{\mcH_{\mcQ}}&\leq \left(1+\frac{1}{a}\right)\frac{2S}{\underline{L}(\mcQ)}\|z\|^2_{\mcH_{\mcQ}}.
    \end{split}
\end{equation}
\end{proof}

\section{Approximation of the heat semigroup on quantum graphs}\label{sec:approxheatsgr}

First we introduce the concept of rooted quantum graphs, following \cite[Sec.~3]{AISW2021}.

\begin{defi}
A \emph{rooted quantum graph} $\mrQ =(\mG, W,\beta, U, x_0)$ is a quantum graph $\mcQ =(\mG, W,\beta, U)$ together with a marked point $x_0\in\mG$ called the \emph{root}. We often denote $\mrQ =(\mcQ, x_0)$. Note that, by the definition of the metric graph in Section \ref{sec:heatsgronquantumgraphs}, the points of $\mG$ correspond to interior points of the edges; in particular, the root is never a vertex. 
\end{defi}

Given a rooted quantum graph $\mrQ =(\mcQ, x_0)$, we may build from it a new quantum graph, by adding at $x_0$ a new vertex $\mv_{x_0}$ with continuity and Kirchhoff's boundary conditions. More precisely, we introduce the following definition, see \cite[Def.~3.2]{AISW2021}.

\begin{defi}\label{defi:Qx0}
Let $\mrQ =(\mcQ, x_0)$ be a rooted quantum graph with $x_0\in \me_0$. We denote by $\mcQ^{x_0}$ the quantum graph such that
\[\mcQ^{x_0} \coloneqq (\mG^{x_0},W^{x_0}, \beta^{x_0}, U^{x_0} ) ,\]
where $\mG^{x_0}=(\mV^{x_0},\mE^{x_0},L^{x_0})$ with
\begin{itemize}
	\item $\mV^{x_0}=\mV\sqcup\{\mv_{x_0}\}$;
	\item $\mE^{x_0}=\left(\mE\setminus\{\me_0\}\right)\cup\left\{\{o(\me_0),\mv_{x_0}\},\{\mv_{x_0},t(\me_0)\}\right\}$, where $\me_0=\{o(\me_0),t(\me_0)\}$ is the edge in $\mG$ containing the coordinate $x_0$; 
	\item $L^{x_0}_{\me}=L_{\me}$ for $\me\in\mE\setminus\{\me_0\}$, \[L^{x_0}_{\{o(\me_0),\mv_{x_0}\}}=x_0,\quad L^{x_0}_{\{\mv_{x_0},t(\me_0)\}}=L_{\me_0}-x_0;\]
	\item $W^{x_0}_{\me}=W_{\me}$ for $\me\in\mE\setminus\{\me_0\}$, \[W^{x_0}_{\{o(\me_0),\mv_{x_0}\}}=W_{\me_0}|_{[0,x_0]},\quad W^{x_0}_{\{\mv_{x_0},t(\me_0)\}}=W_{\me_0}|_{[x_0,L_{\me_0}]};\]
    \item $(\beta^{x_0})^{\mv}=\beta^{\mv}$ for all $\mv \in\mV$, $(\beta^{x_0})^{\mv_{x_0}}(1)=(\mv_{x_0},o(\me_0))$, $(\beta^{x_0})^{\mv_{x_0}}(2)=(\mv_{x_0},t(\me_0))$;
	\item $U^{x_0}_{\mv}=U_{\mv}$ for $\mv\in\mV$, and
    \[U^{x_0}_{\mv_{x_0}}=\begin{pmatrix}
        0 & 1\\ 1 & 0
    \end{pmatrix}\]
\end{itemize}  
which, by \eqref{eq:UvconKN}, represents the Kirchhoff-condition in $\mv_{x_0}$.  We denote by $G^{x_0}=(\mV^{x_0},\mE^{x_0})$ the new combinatorial graph.
\end{defi}

For the fully precise bond-level formulation of the above definition we refer to \cite[Def.~3.2]{AISW2021}.\\

Let $\mrQ =(\mG, W,\beta, U, x_0)=(\mcQ,x_0)$ with $\mG=(\mV,\mE, L)$ be an infinite rooted quantum graph which we will now fix throughout the section. Consider the quantum graph $\mcQ^{x_0}$ associated to it and its rescaling from Definition \ref{defi:mcHQ}. 
To simplify notations, in this section we denote
\begin{equation}
    \mcQ\coloneqq \mcQ^{x_0},\quad \mcH_{\mcQ}\coloneqq \mcH,\quad L_{\mcQ}^2(\me,w)\coloneqq L^2(\me),\quad H_{\mcQ}^2(\me,w)\coloneqq H^2(\me).
\end{equation}

Notice that $\mcQ$ has a marked vertex $\mv_{x_0}$ which can be considered as the root of the combinatorial graph $\mG^{x_0}$. Let $\left(A_{\mcQ},\mcD(A_{\mcQ})\right)$ be the associated operator to $\mcQ$ from \eqref{eq:opAQ}-\eqref{eq:domopAQ}.  Notice, that defining the operator in this way, we always assume the edge coordinates to vary in $[0,1]$ and the parameters $L_{\me}$ be contained in the operator.

\begin{defi}\label{defi:rballs}
For every $r\in\nat$, we define the $r$-ball around $\mv_{x_0}$ as 
\begin{equation}\label{eq:Bx0r}
    \mB(x_0, r) \coloneqq \{\mw\in \mV^{x_0}:\rho_{G^{x_0}}(\mv_{x_0}, \mw) \leq r\}.
\end{equation} 
(To simplify notations, we write $x_0$ instead of $\mv_{x_0}$.) Here $\rho_{G^{x_0}}(\mv, \mw)$ is the length of the shortest path connecting the vertices $\mv$ and $\mw$ in the discrete graph $G^{x_0}$. Denote by $\mE(x_0, r)$ the set of edges connecting two vertices from $\mB(x_0, r)$ in $G^{x_0}$, and let
\begin{equation}\label{eq:rballmetricgraph}
\mG(x_0,r)
\end{equation}
be the metric subgraph of $\mG^{x_0}$ induced by the vertices in $\mB(x_0,r)$ and the edges in $\mE(x_0,r)$.
By Assumption \ref{ass:main}, $\mG(x_0,r)$ is finite.
\end{defi}

For each $r\in\nat$, we are going to define a new quantum graph $\mcQ_r$  as follows. 
\begin{defi}
Let $r\in\nat$ arbitrary. Define
\begin{equation}
W_{r,\me}\coloneqq W_{\me}, \text{ if }\me\in\mE(x_0,r),
\end{equation}
a potential on $\mG(x_0,r)$, 
\begin{equation}\label{eq:betavr}
\beta_{r}^{\mv}\coloneqq \begin{cases}
\beta^{\mv}, &\text{ if } \mv\in \mB(x_0,r-1),\\
\text{any labeling of the edge set }\mE(x_0,r)\cap\mE_{\mv},&\text{ if } \mv\in \mB({x_0},r)\setminus \mB({x_0},r-1),
\end{cases}
\end{equation}
and
\begin{equation}
U_{r,\mv}\coloneqq \begin{cases}
U_{\mv}, &\text{ if } \mv\in \mB(x_0,r-1),\\
KN_{r,\mv},&\text{ if } \mv\in \mB({x_0},r)\setminus \mB({x_0},r-1),
\end{cases}
\end{equation}
where $KN_{r,\mv}$ is the unitary matrix of size $|\mE(x_0,r)\cap\mE_{\mv}|$ representing the continuity and Kirchhoff-conditions in a vertex $\mv\in \mB({x_0},r)\setminus \mB({x_0},r-1)$, for the edges $\mE(x_0,r)\cap\mE_{\mv}$, cf.~\eqref{eq:UvconKN}.
Thus, $U_{r,\mv}$ is a unitary matrix for each $\mv\in \mB(x_0,r)$.

We remark that since we require continuity and Kirchhoff--Neumann condition in the boundary vertices of the $r$-ball $\mG(x_0,r)$,  we could choose any labeling of the edges incident to such vertices in \eqref{eq:betavr}, cf.~\cite{AISW2021}, Remark before Section 3.\\

Define the quantum graph
\begin{equation}\label{eq:Qrdef}
\mcQ_r\coloneqq \left(\mG(x_0,r),W_r,\beta_r,U_r\right)
\end{equation}
with
\begin{equation}
    W_r=\left(W_{r,\me}\right)_{\me\in\mE(x_0,r)},\quad\beta_r= \left(\beta_{r}^{\mv}\right)_{\mv\in \mB(x_0,r)},\quad U_r=\left(U_{r,\mv}\right)_{\mv\in \mB(x_0,r)}.
\end{equation}
\end{defi}

We will associate a Hamiltonian operator to $\mcQ_r$ using $A_{\mcQ}$. Define the spaces
\begin{equation}\label{eq:mcHr}
\begin{split}
\mcH_r&\coloneqq \bigoplus_{\me\in\mE(x_0,r)}L^2(\me)\\ 
\|f\|^2_{\mcH_r}&\coloneqq\sum_{\me\in\mE(x_0,r)}\|f_{\me}\|^2_{L^2(\me)},\quad f\in\mcH_r
\end{split}
\end{equation}
and
\begin{equation}\label{eq:mcH2r}
\begin{split}
\mcH^2_r&\coloneqq \bigoplus_{\me\in\mE(x_0,r)}H^2(\me)\\ 
\|f\|^2_{\mcH^2_r}&\coloneqq\sum_{\me\in\mE(x_0,r)}\|f_{\me}\|^2_{H^2(\me)}, \quad f\in\mcH^2_r.
\end{split}
\end{equation}
Denote by
\begin{equation}\label{eq:A1rvA2rv}
A_{r,\mv}=\mi \left(U_{r,\mv}-\Id\right),\;  B_{r,\mv}=U_{r,\mv}+\Id,
\end{equation}
cf.~\eqref{eq:A1vA2v}. For $\mv\in\mB(x_0,r)$ we abbreviate $\mE^r_{\mv}\coloneqq\mE(x_0,r)\cap\mE_{\mv}$, and define
\begin{equation}
\begin{split}
(A_{\mcQ_r}z)_{\me}&=\frac{1}{L_{\me}^2}z_{\me}''-W_{\me}z_{\me},\text{ for } \me\in\mE({x_0},r),\\
\mcD(A_{\mcQ_r})&=\left\{z\in\mcH^2_r:\text{ for all }\mv\in \mB(x_0,r),\,A_{r,\mv}Z(\mv)+B_{r,\mv}\cdot \mathrm{diag}\left(\frac{1}{L_{\me}}\right)_{\me\in\mE^r_{\mv}} \cdot Z'(\mv)=0\right\},\label{eq:bdrycondBr}
\end{split}
\end{equation}
see \eqref{eq:opAQ}-\eqref{eq:domopAQ}.
Notice that the boundary conditions coincide with the original ones from $\mcD(A_\mcQ)$ in the vertices of $ \mB(x_0,r-1)$, while they are the continuity and weighted Kirchhoff--Neumann conditions in the vertices of $ \mB(x_0,r)\setminus \mB(x_0,r-1)$.\\

The next result can be proved in the same way as Proposition \ref{prop:AQgen}.

\begin{prop}\label{prop:AQrgen}
Let $r\in\nat$ be arbitrary. The operator $\left(A_{\mcQ_r},\mcD(A_{\mcQ_r})\right)$ on $\mcH_r$ 
is the operator associated with the form
\begin{equation}\label{eq:formAQr}
\begin{split}
\ea_{\mcQ_r}(z,h)&=\sum_{\me\in\mE(x_0,r)}\int_0^{1}\left(\frac{1}{L_{\me}^2}z_{\me}'(x)\overline{h_{\me}'(x)}+W_{\me}(x)z_{\me}(x)\overline{h_{\me}(x)}\right)L_{\me}\dx\\
&+\sum_{\mv\in \mB(x_0,r-1)}\langle\Lambda_{\mv}P_{R,\mv}Z(\mv),P_{R,\mv}H(\mv)\rangle ,\\
\mcD(\ea_{\mcQ_r})&=\left\{z\in \bigoplus_{\me\in \mE(x_0,r)} H^1(\me)\colon P_{D,\mv}Z(\mv)=0\text{ for all }\mv\in\mB(x_0,r)\right\}.
\end{split}
\end{equation}
Hence, $\left(A_{\mcQ_r},\mcD(A_{\mcQ_r})\right)$ is densely defined, self-adjoint and bounded from above by the same constant as $\left(A_{\mcQ},\mcD(A_{\mcQ})\right)$. Furthermore, it generates a strongly continuous semigroup $(S_{\mcQ_r}(t))_{t\geq 0}$ on $\mcH_r$ which is quasicontractive and analytic, and its norms satisfy the same exponential bound as those of $(S_{\mcQ}(t))_{t\geq 0}$.

As a consequence, the abstract Cauchy problem 
\begin{equation}
\begin{split}
\dot{z}(t)&=A_{\mcQ_r}z(t),\\
z(0)&=z_0.
\end{split}
\end{equation}
is well-posed on $\mcH_r$ and its unique mild solutions are given by the orbits of the semigroup $S_{\mcQ_r}(t)z_0,$ $t\geq 0$ for each $z_0\in\mcH_r.$
\end{prop}

Using a Trotter--Kato type approximation result of Ito and Kappel \cite{ItoKappel98}, we are going to show that the semigroups $(S_{\mcQ_r}(t))_{t\geq 0}$ converge strongly to $(S_{\mcQ}(t))_{t\geq 0}$, if $r\to+\infty$. For this purpose we introduce the following operators.

\begin{defi}\label{defi:PrJr}
Let $r\in\nat$ be arbitrarily fixed and define the linear operators
\begin{enumerate}
	\item $P_r\colon \mcH\to\mcH_r$, $P_rf\coloneqq \left(f_{\me}\right)_{\me\in\mE(x_0,r)}$, $f\in\mcH$;
	\item $J_r\colon \mcH_r\to\mcH$, 
	\[(J_rg)_{\me}\coloneqq \begin{cases} g_{\me}, & \text{ if }{\me\in\mE(x_0,r)},\\
	0, &\text{ if }\me\in\mE\setminus \mE(x_0,r),\end{cases}\] 
	$g\in \mcH_r.$
\end{enumerate}
\end{defi}
The operators $P_r$ and $J_r$, $r\in\nat$ satisfy the next properties, cf.~\cite[(A1)--(A3)]{ItoKappel98}.
\begin{prop}\label{prop:PrJr}
Let $P_r$ and $J_r$, $r\in\nat$ be the operators from Definition \ref{defi:PrJr}. Then
\begin{enumerate}
	\item For each $r\in\nat$, $P_r$ and $J_r$ are bounded linear operators with $\|P_r\|\leq 1$, $\|J_r\|=1.$
	\item For each $f\in\mcH$, $\|J_rP_rf-f\|_{\mcH}\to 0$, $r\to+\infty$.
	\item For each $r\in\nat$, $P_rJ_r=\Id_{\mcH_r}$.
\end{enumerate}
\end{prop}
\begin{proof}
(1) and (3) are direct consequences of Definition \ref{defi:PrJr}. To verify (2), let $f\in\mcH$ be arbitrary. Then
\begin{equation}
   \|J_rP_rf-f\|_{\mcH}^2=\sum_{\me\in \mE\setminus \mE(x_0,r)}\|f_{\me}\|^2_{L^2(\me)}.
\end{equation}
We know that
\begin{equation}  \|f\|_{\mcH}^2=\sum_{\me\in \mE}\|f_{\me}\|^2_{L^2(\me)}<\infty.
\end{equation}
Furthermore, the connectedness of $\mG$ implies that if $\me\in\mE$ is arbitrary, for sufficiently large $r>0$, $\me\in \mE(x_0,r)$ holds. From these facts we have
\begin{equation}
   \|J_rP_rf-f\|_{\mcH}^2=\sum_{\me\in \mE\setminus \mE(x_0,r)}\|f_{\me}\|^2_{L^2(\me)}\to 0,\quad r\to+\infty.
\end{equation}
\end{proof}
Thus, we are ready to state the result about the strong convergence of the semigroups on the $r$-balls to the semigroup on the quantum graph.
\begin{theo}\label{thm:heatsgrconv} 
For the semigroups $(S_{\mcQ}(t))_{t\geq 0}$ and $(S_{\mcQ_r}(t))_{t\geq 0}$, generated by $(A_{\mcQ},\mcD(A_{\mcQ}))$ and $(A_{\mcQ_r},\mcD(A_{\mcQ_r}))$, $r\in\nat$, the following convergence holds true: for any $f\in\mcH$ and $t\geq 0$,
\begin{equation}
\|J_rS_{\mcQ_r}(t)P_rf-S_{\mcQ}(t)f\|_{\mcH}\to 0,\text{ as }r\to+\infty,
\end{equation}
uniformly in $t$ on compact intervals.
\end{theo}
\begin{proof}
We are going to use \cite[Prop.~3.1]{ItoKappel98}. By Propositions \ref{prop:AQrgen} and \ref{prop:PrJr},  it is enough to prove the following:

\textbf{(C1)} There exists a subset $D\subset\mcD(A_{\mcQ})$ which is dense in $\mcH$ and $\overline{(\la\cdot \Id-A_{\mcQ})D}=\mcH$ for $\la>0.$

\textbf{(C2)} For all $z\in D$, there exists a sequence $(z_r)_{r\in\nat}$ with $z_r\in\mcD(A_{\mcQ_r})$, $r\in\nat$ such that
\begin{equation}\label{eq:JrArlimit}
\lim_{r\to\infty}J_rz_r=z,\quad \lim_{r\to\infty}J_rA_{\mcQ_r} z_r=A_{\mcQ}z.
\end{equation} 

By \cite[Prop.~3.1]{ItoKappel98}, these facts imply the claim.\\

We first prove \textbf{(C1)}. Let $D\coloneqq \mcD(A_{\mcQ})$. The Feller-Miyadera-Phillips Theorem \cite[Thm.~II.3.8]{EN00} implies that $D$ is dense in $\mcH$ and for $\la$ big enough, $\la\in\rho(A_{\mcQ})$ holds, hence $\overline{(\la\cdot\Id-A_{\mcQ})D}=\mcH$ is satisfied.\\

To prove \textbf{(C2)}, let $z\in \mcD(A_{\mcQ})$ and $\ve>0$ be arbitrary fixed. We are going to construct  functions $z_r\in \mcD(A_{\mcQ_r})$, $r\in\nat$ such that  for some constant $L>0$,
\begin{equation}\label{eq:C2estimate}
    \|z-J_rz_r\|^2_{\mcH}+\|A_{\mcQ}z-J_rA_{\mcQ_r}z_r\|^2_{\mcH}<L\cdot \ve
\end{equation}
holds, if $r$ is big enough.\\

Since $z\in\mcH^2$, there exists $R>0$ such that if $r>R$, then
\begin{equation}\label{eq:znormrest}
    \sum_{\me\in\mE\setminus \mE(x_0,r-1)}\|z_{\me}\|^2_{H^2(\me)}<\ve.
\end{equation}
We are going to define $z_r\in \mcD(A_{\mcQ_r})$ for $r>R$. Let
\begin{equation}\label{eq:zreqzonErminus1}
    (z_r)_{\me}=z_{\me},\quad \me\in \mE(x_0,r-1).
\end{equation}

If $\me\in \mE(x_0,r)\setminus \mE(x_0,r-1)$ holds, and both endpoints of $\me$ are contained in $\mB(x_0,r)$, we define $(z_r)_{\me}\equiv 0$. Otherwise, assume that $o(\me)\in \mB(x_0,r-1)$ and define
\begin{equation}\label{eq:zrdef}
    (z_r)_{\me}(x)\coloneqq \begin{cases}
        z_{\me}(x), &\text{ if }x\in[0,\frac{1}{2}],\\
        p_{\me}(x), &\text{ if }x\in[\frac{1}{2},1],
    \end{cases}
\end{equation}
where $p_{\me}(x)$ is the third-degree Hermite interpolation polynomial with the data points
\begin{equation}\label{eq:Hermitedata}
\begin{split}
   p_{\me}\left(\frac{1}{2}\right)&=z_{\me}\left(\frac{1}{2}\right), \quad p_{\me}'\left(\frac{1}{2}\right)=z'_{\me}\left(\frac{1}{2}\right),\\
   p_{\me}(1)&=0,\quad 
   p'_{\me}(1)=0,
\end{split}
\end{equation}
that is
\begin{equation}\label{eq:Hermiteinterpolpolynom}
    p_{\me}(x)=\alpha(x)\cdot z_{\me}\left(\frac{1}{2}\right)+\beta(x)\cdot z'_{\me}\left(\frac{1}{2}\right)
\end{equation}
with given basis polynomials $\alpha,\beta$ of third degree.

For $\me\in \mE(x_0,r)\setminus \mE(x_0,r-1)$ observe that by \eqref{eq:zrdef} $(z_r)_{\me}'$ exists almost everywhere and by \eqref{eq:Hermitedata},  $(z_r)_{\me}'\in H^1(0,1)$ holds. Thus, by \eqref{eq:zreqzonErminus1}, $(z_r)_{\me}\in H^2(\me)$ is true  for $\me\in \mE(x_0,r)$. Definitions \eqref{eq:zreqzonErminus1} and \eqref{eq:zrdef} assure that in the vertices of $\mB(x_0,r-1)$ the conditions in the domain of $A_{\mcQ_r}$ are satisfied. In the vertices $\mv$ of $\mB(x_0,r)\setminus \mB(x_0,r-1)$ the boundary conditions \eqref{eq:bdrycondBr} in $\mcD(A_{\mcQ_r})$ are trivially satisfied since either $(z_r)_{\me}\equiv 0$ or 
\[(z_r)_{\me}(1)=p_{\me}(1)=0,\quad  (z_r)_{\me}'(1)=p'_{\me}(1)=0\] for all edges $\me$ incident to $\mv=t(\me)$. Thus, $z_r\in \mcD(A_{\mcQ_r})$ holds.\\

Using Sobolev embedding, we have that for each $\me\in\mE$, $z_{\me}\in C^1[0,1]$ holds, furthermore, by \eqref{eq:znormrest}, there exists $k>0$ (independent of $r$) such that
\begin{equation}
    \sum_{\me\in\mE(x_0,r)\setminus \mE(x_0,r-1)}\left(\|z_{\me}\|^2_{\infty}+\|z'_{\me}\|^2_{\infty}\right)<k\cdot\sum_{\me\in\mE(x_0,r)\setminus \mE(x_0,r-1)}\|z_{\me}\|^2_{H^2(\me)}<k\cdot \ve.
\end{equation} 
Using this, \eqref{eq:znormrest} and \eqref{eq:Hermiteinterpolpolynom}, we obtain that there exists $C>0$ (depending only on $\alpha$ and $\beta$) such that
\begin{equation}\label{eq:zrnormonErminusErminus1}
\begin{split}
   \sum_{\me\in\mE(x_0,r)\setminus \mE(x_0,r-1)}\|(z_r)_{\me}\|^2_{H^2(\me)}&\leq 2\cdot\sum_{\me\in\mE(x_0,r)\setminus \mE(x_0,r-1)} \left(\|z_{\me}\|^2_{H^2(\me)}+ \|p_{\me}\|^2_{H^2(\me)}\right)\\&<2\cdot(\ve+k\cdot\ve\cdot C)=\ve\cdot 2(1+k\cdot C)\eqqcolon c\cdot \ve.
   \end{split}
\end{equation}

By the definition of $A_{\mcQ}$ \eqref{eq:opAQ} and $A_{\mcQ_r}$ \eqref{eq:bdrycondBr}, and by Assumption \ref{ass:main}, there exists $K>0$ (independent of $r$) such that 
\begin{equation}
        \|z-J_rz_r\|^2_{\mcH}+\|A_{\mcQ}z-J_rA_{\mcQ_r}z_r\|^2_{\mcH}\leq K\cdot \|z-J_rz_r\|^2_{\mcH^2}.
\end{equation}
Notice that this estimate and \eqref{eq:znormrest}, \eqref{eq:zreqzonErminus1}, \eqref{eq:zrnormonErminusErminus1} imply
\begin{equation}
\begin{split}
        &\|z-J_rz_r\|^2_{\mcH}+\|A_{\mcQ}z-J_rA_{\mcQ_r}z_r\|^2_{\mcH} \\
        &\leq K\cdot \left(\sum_{\me\in \mE(x_0,r)\setminus \mE(x_0,r-1)}\left(\|z_{\me}\|^2_{H^2(\me)} + \|(z_r)_{\me}\|^2_{H^2(\me)}\right)+\sum_{\me\in\mE\setminus \mE(x_0,r)}\|z_{\me}\|^2_{H^2(\me)}\right)\\
        &<K\cdot \left(\ve+c\cdot \ve \right).
\end{split}
\end{equation}
Since $\ve>0$ was arbitrary, \eqref{eq:C2estimate}, and thus \textbf{(C2)}, is satisfied.

\end{proof}

\section{Benjamini--Schramm convergence of heat semigroups on quantum graphs}\label{sec:BSlimitquantumgraphs}

In this section we first follow \cite[Sec.~3]{AISW2021}. We introduce an equivalence relation on quantum graphs and define a distance between the equivalence classes.

\begin{defi}\label{defi:rootedgraph}
We call $(G,\mv)$ a \emph{rooted discrete graph} if $G$ is a discrete graph and $\mv$ is a marked vertex in $G$. If $(G,\mv)$ and $(\wG,\hat{\mv})$ are two rooted discrete graphs, we denote
\[\phi\colon (G,\mv)\xrightarrow{\sim} (\wG,\wmv)\]
if $\phi$ is a graph isomorphism -- that is, an adjacency preserving bijection between the vertices of $G$ and $\wG$ -- such that $\phi(\mv)=\wmv$.
\end{defi}

\begin{defi}\label{defi:Qisomorph}
Let $\mrQ =(\mG, W,\beta, U, x_0)=(\mcQ, x_0)$ and $\wmrQ =(\wmG, \wW,\wbeta, \wU, \wx)=(\wmcQ, \wx)$ be two rooted quantum graphs  with $\mG=(\mV,\mE,L)$ and $\wmG=(\wmV,\wmE,\wL)$, respectively. We say that $\mrQ$ and $\wmrQ$ are \emph{equivalent} if there exists a graph isomorphism
\[\phi: \left(G^{x_0},\mv_{x_0}\right)\xrightarrow{\sim} \left(\wG^{\wx},\mv_{\wx}\right)\] satisfying 
\[L^{x_0}=\wL^{\wx}\circ\phi,\, W^{x_0}=\wW^{\wx}\circ\phi,\,\phi\circ\beta^{x_0}=\wbeta^{\wx},\, U^{x_0}=\wU^{\wx}\circ\phi.\]
We denote the equivalence class of $\mrQ$ by $[\mrQ]$ and by
\begin{equation}\label{eq:Qstardefi}
\mbQs\coloneqq\left\{[\mrQ]:\mrQ\text{ is a rooted quantum graph}\right\}
\end{equation}
the set of equivalence classes.
\end{defi}

\begin{defi}[Equation (3.1) of \cite{AISW2021}]\label{defi:metricrooted}
Let $\mrQ=(\mcQ, x_0) =(\mG, W,\beta, U, x_0)$ and $\wmrQ=(\wmcQ, \wx) =(\wmG, \wW,\wbeta, \wU, \wx)$ be two rooted quantum graphs with $\mG=(\mV,\mE,L)$ and $\wmG=(\wmV,\wmE,\wL)$, respectively. For any $k\in\nat$, denote by $\mG(x_0,k)$ and  $\wmG(\wx,k)$ the $k$-balls around the roots of $\mrQ$ and $\wmrQ$, respectively, see \eqref{eq:Bx0r}. We define a pseudometric between them as follows
\begin{equation}\label{eq:metricrooted}
d((\mcQ, x_0),(\wmcQ, \wx))=\frac{1}{1+\alpha},
\end{equation}
where
\begin{align} \alpha=\sup&\left\{r>0 :\exists \phi: \mG(x_0, \left \lfloor{r}\right \rfloor )\xrightarrow{\sim} \wmG(\wx, \left \lfloor{r}\right \rfloor)\text{ with }\phi\circ\beta^{x_0}=\wbeta^{\wx},\right.\\
&\left. \delta_{\left \lfloor{r}\right \rfloor ,\phi}\left((L^{x_0}, W^{x_0}, U^{x_0}), (\wL^{\wx}, \wW^{\wx}, \wU^{\wx})\right)<\frac{1}{r}\right\}\end{align}
with
\begin{align}
\delta_{k,\phi}&\left((L^{x_0}, W^{x_0}, U^{x_0}), (\wL^{\wx}, \wW^{\wx}, \wU^{\wx})\right)=\max\left\{\max_{\me\in\mE(x_0,k)}\left|L^{x_0}_{\me}-\wL^{\wx}_{\phi(\me)}\right|,\right.\\
&\left.\max_{\me\in\mE(x_0,k)}\sup_{t\in [0,1]}\left|W^{x_0}_{\me}(tL^{x_0}_{\me})-\wW^{\wx}_{\phi(\me)}(t\wL^{\wx}_{\phi(\me)})\right|,\max_{\mv\in \mB(x_0,k)}\left\|U^{x_0}_{\mv}-\wU^{\wx}_{\phi(\mv)}\right\|\right\},\quad k\in\nat,
\end{align}
where $\mE(x_0,k)$ denotes the edges connecting the vertices of $\mB(x_0,k)$ in $G^{x_0}$, cf.~Definition \ref{defi:rballs}.
\end{defi}

\begin{rem}\label{rem:distKN}
    Assume that  in the above definition, for the quantum graphs $\mcQ$, $\wmcQ$, in all vertices  continuity and Kirchhoff-Neumann conditions hold, cf.~\eqref{eq:wcontcond}, \eqref{eq:wKNcond}. Then the metric $d$ simplifies to the following one:
\begin{equation}
d((\mcQ, x_0),(\wmcQ, \wx))=\frac{1}{1+\gamma}
\end{equation}
where
\begin{align} \gamma=\sup&\left\{r>0 :\exists \phi: \mG(x_0, \left \lfloor{r}\right \rfloor )\xrightarrow{\sim} \wmG(\wx, \left \lfloor{r}\right \rfloor)\text{ with }\phi\circ\beta^{x_0}=\wbeta^{\wx},\right.\\
&\left. \varrho_{\left \lfloor{r}\right \rfloor ,\phi}\left((L^{x_0}, W^{x_0}), (\wL^{\wx}, \wW^{\wx})\right)<\frac{1}{r}\right\}\end{align}
with
\begin{align}
\varrho_{k,\phi}\left((L^{x_0}, W^{x_0}), (\wL^{\wx}, \wW^{\wx})\right)=\max&\left\{\max_{\me\in\mE(x_0,k)}\left|L^{x_0}_{\me}-\wL^{\wx}_{\phi(\me)}\right|,\right.\\
&\left.\max_{\me\in\mE(x_0,k)}\sup_{t\in [0,1]}\left|W^{x_0}_{\me}(t)-\wW^{\wx}_{\phi(\me)}(t)\right|\right\},\quad k\in\nat,
\end{align}
where $\mE(x_0,k)$ denotes the edges connecting the vertices of $\mB(x_0,k)$ in $G^{x_0}$, cf.~Definition \ref{defi:rballs}.
\end{rem}

\begin{rem}\label{rem:Qstarsepcomplete}
The pseudometric $d$ defined in \eqref{eq:metricrooted} extends naturally to the set of equivalence classes $\mbQs$. By \cite[Lem.~3.4]{AISW2021}, we also have that $(\mbQs,d)$ is a separable complete metric space. 
\end{rem}

From this point on, all quantum graphs under consideration carry continuity and Kirchhoff vertex conditions. The corresponding vertex matrices $U_{\mv}=\frac{2}{d(\mv)}\mathbf{1}-\Id$, see \eqref{eq:UvconKN}, are real \emph{symmetric} -- equivalently, the vertex conditions have the real projection form \eqref{eq:KNprojections} of Remark \ref{rem:contwKNcond} -- and the potential $W$ is real-valued. Hence $A_{\mcQ}$, consequently the members of the semigroup $\left(S_{\mcQ}(t)\right)_{t\geq 0}$  are \emph{real} operators. 
For the rest of the paper we therefore work with real-valued functions: all function spaces are understood over $\real$, and the results of Sections \ref{sec:heatsgronquantumgraphs} and \ref{sec:approxheatsgr} apply by restriction to the real subspace.\vspace{0.5cm}

For $D\in\nat$, $0<m\leq M$ define the subset
\begin{equation}
\mbKs\subset \mbQs
\end{equation}
as the subset of equivalence classes $[\mcQ,x_0]=[\mG, W,\beta, U, x_0]$ such that \begin{equation}\label{eq:QstarDmMK}
\begin{split}
&U_{\mv}\text{ represents the continuity and Kirchhoff's}\\
&\text{conditions for each }\mv\in\mV\text{  (cf.~\eqref{eq:UvconKN})},\\
& d(\mv)\leq D\text{ for all }\mv\in\mV,\\
& m\leq \underline{L}(\mcQ)\leq \overline{L}(\mcQ)\leq M,\\
&W_{\me}\in \Lip([0,1]),\; \max\{\|W_{\me}\|_{\infty}, \Lip(W_{\me}) \} \leq M,\,\forall \me\in\mE.
\end{split}
\end{equation}
Here $\Lip(I)$ denotes the set of Lipschitz-continuous functions on the interval $I$ and $\Lip(f)$ is the Lipschitz-constant of $f$. It is clear that unitary matrices representing the continuity and Kirchhoff's conditions are invariant under the isomorphism from Definition \ref{defi:Qisomorph}, hence, $\mbKs$ is a well-defined subset of $\mbQs$.

\begin{rem}
    Combining arguments from the proofs of \cite[Lem.~3.4 and 3.6]{AISW2021}, one shows that $\mbKs$ is a closed subset of $(\mbQs,d)$. 
\end{rem}

Below we will need to interpret the semigroup operators associated with different quantum graphs to act on a common Hilbert space $\mcH$. Hence, we define
 \begin{equation}\label{eq:mcHell2}
     \mcH\coloneqq \ell^2(L^2(0,1))
 \end{equation}
with the norm
\begin{equation}
    \sum_{j=1}^{\infty}\|f_j\|^2_{L^2(0,1)}\eqqcolon \|f\|^2_{\mcH}.
\end{equation}
Similarly, we define
\begin{equation}\label{eq:mcH2ell2}
     \mcH^2\coloneqq \ell^2(H^2(0,1))
 \end{equation}
with the norm
\begin{equation}
    \sum_{j=1}^{\infty}\|f_j\|^2_{H^2(0,1)}\eqqcolon \|f\|^2_{\mcH^2}.
\end{equation}

\begin{notat}\label{notat:mrQmcQx}
Let $\mrQ=(\mcQ,x_0)$ be a rooted quantum graph satisfying \eqref{eq:QstarDmMK}, which can be finite or infinite. In the following, the notations
\begin{equation}
    \mcH_{\mrQ},\quad \mE_{\mrQ},\quad A_{\mrQ}, \quad \left(S_{\mrQ}(t)\right)_{t\geq 0}
\end{equation}
denote the Hilbert space, edge set, generator and semigroup from Definition \ref{defi:mcHQ}, \ref{defi:opAQ}, Proposition \ref{prop:AQgen}, respectively, associated with the quantum graph $\mcQ^{x_0}$, cf.~Definition \ref{defi:Qx0}.
\end{notat}

We are going to construct a bounded linear map
\begin{equation}\label{eq:defPsiQ}
    \Psi_{\mrQ}:\mcH\to\mcH_{\mrQ},\quad \Psi_{\mrQ}f=f_{\mrQ}\in\mcH_{\mrQ},
\end{equation}
such that $\Psi_{\mrQ}$ matches coordinates of $f$ one by one to coordinates of $f_{\mrQ}$, and is compatible with the equivalence relation from Definition \ref{defi:Qisomorph}.
For this purpose, we define an injective map 
\begin{equation}\label{defi:psiQ}
    \psi_{\mrQ}\colon \mE_{\mrQ}\to\nat
\end{equation}
which numbers the edges of $\mcQ^{x_0}$ by consecutive natural numbers. 

We define the map $\psi_{\mrQ}$ via a breadth-first search (BFS) traversal of $\mG^{x_0}$ started at the root $\mv_{x_0}$, in which the edges incident to a vertex are always explored in the order given by the labeling maps $\beta^{\mv}$ from \eqref{eq:betav}. 
First, the BFS traversal yields a linear discovery order on the vertices of $\mG^{x_0}$. Let $\mv_0\coloneqq \mv_{x_0}$ and process the vertices in a first-in-first-out queue: when a vertex $\mv$ is processed, its outgoing oriented edges are visited in the order $\beta^{\mv}(1),\dots,\beta^{\mv}(d(\mv))$, and those of their endpoints that have not been discovered so far are appended to the queue in this order. Since $\mG^{x_0}$ is connected and locally finite, this yields an enumeration $\mv_0,\mv_1,\mv_2,\dots$ of the vertices of $\mG^{x_0}$ which lists the balls $\mB(x_0,k)$, $k\in\nat$, consecutively.

Using this order, we assign consecutive natural numbers $\psi_{\mrQ}(\me)$ starting with $1$ to the edges $\me\in\mE_{\mrQ}(x_0,1)$ of the (finite) $1$-ball around $\mv_{x_0}$ in $\mcQ^{x_0}$. After we continue with the numbering of the edges in $\mE_{\mrQ}(x_0,2)\setminus \mE_{\mrQ}(x_0,1)$ starting with $|\mE_{\mrQ}(x_0,1)|+1$, etc. Within each set $\mE_{\mrQ}(x_0,k)\setminus \mE_{\mrQ}(x_0,k-1)$ (where $\mE_{\mrQ}(x_0,0)\coloneqq\emptyset$), the edges are ordered as follows: to each edge $\me$ we assign the pair $(i,j)$, where $\mv_i$ is the endpoint of $\me$ with the smaller BFS index and $j$ is the label of $\me$ at $\mv_i$, that is, $\beta^{\mv_i}(j)=\me$. Then we order the edges lexicographically by these pairs. For more details see \cite[Def.~1.24.]{Hofstad2024} and \cite[Sec.~3]{AMR07}.\\

The connectedness of $\mG$ implies that if $\me\in\mE_{\mrQ}$ is arbitrary, for sufficiently large $r>0$, $\me\in \mE_{\mrQ}(x_0,r)$ holds - hence, each edge attains a number in this way. If the edge $\me$ has number $\psi_{\mrQ}(\me)\in \nat$, we define
\begin{equation}\label{eq:Qnumbering}
    f_{\mrQ,\me}\coloneqq f_{\psi_{\mrQ}(\me)},\quad \me\in\mE_{\mrQ}.
\end{equation}
Thus, we obtained a linear map $\Psi_{\mrQ} :\mcH\to\mcH_{\mrQ}$,
\begin{equation}\label{eq:PsiQdef}
\Psi_{\mrQ} f\coloneqq f_{\mrQ}=\left(f_{\mrQ,\me}\right)_{\me\in\mE_{\mrQ}}\in\mcH_{\mrQ}.
\end{equation}

\begin{rem}\label{rem:psiQcompatible}
    Let  $\mrQ=(\mcQ,x_0)$, $\wmrQ=(\wmcQ,\wx)$ be  rooted quantum graphs with
    \begin{equation}
        d(\mrQ,\wmrQ)<\frac{1}{1+r},
    \end{equation}
    hence,
     \begin{equation}
\exists \phi: \wmG(\wx, r )\xrightarrow{\sim} \mG(x_0, r)\text{ with }\beta^{x_0}=\phi\circ\wbeta^{\wx},
\end{equation}
where the label condition is understood for all bonds of the $r$-balls (in particular, at the vertices of $\mB_{\wmrQ}(\wx,r)\setminus\mB_{\wmrQ}(\wx,r-1)$ it refers to the labels of the incident edges lying in the $r$-ball).
Then a straightforward induction along the BFS traversal shows that
\begin{equation}\label{eq:numberingcompatible}
    \psi_{\wmrQ}(\me)=\psi_{\mrQ}(\phi(\me)),\quad \me\in \mE_{\wmrQ}(\wx,r),
\end{equation}
that is, for any $f\in\mcH$,
\begin{equation}\label{eq:fQhatfQ}
    f_{\wmrQ,\me}=f_{\mrQ, \phi(\me)},\quad \me\in \mE_{\wmrQ}(\wx,r)
\end{equation}
holds. 
\end{rem}

\begin{rem}\label{rem:HqL2norm}
     If $\mrQ$ satisfies \eqref{eq:QstarDmMK}, we have 
\begin{equation}\label{eq:mcHmchQnormequiv}  
   \|\Psi_{\mrQ}f\|_{\mcH_{\mrQ}}= \|f_{\mrQ}\|_{\mcH_{\mrQ}}\leq \sqrt{M}\cdot \|f\|_{\mcH},\quad f\in\mcH.
    \end{equation}
    That is, $\Psi_{\mrQ}:\mcH\to\mcH_{\mrQ}$ is bounded with
    \begin{equation}\label{eq:PsiQnormestimate1}
        \|\Psi_{\mrQ}\|\leq \sqrt{M}.
    \end{equation}
 \end{rem}
It is clear that if the number of edges in $\mrQ$ is infinite, $\Psi_{\mrQ}$ is an isomorphism with
\begin{equation}
    \|\Psi_{\mrQ}^{-1}f_{\mrQ}\|_{\mcH}\leq \frac{1}{\sqrt{m'}}\cdot\|f_{\mrQ}\|_{\mcH_{\mrQ}},\text{ hence }\|\Psi_{\mrQ}^{-1}\|\leq \frac{1}{\sqrt{m'}},
\end{equation}
where
\begin{equation}\label{eq:rootmargin}
    m'=m'(\mrQ)\coloneqq\min\{\rho(\mrQ),\,m\}\text{ with }
    \rho(\mrQ)\coloneqq\min\{x_0,\,L_{\me_0}-x_0\}.
\end{equation}
We call $\rho(\mrQ)$ the \emph{root margin} of $\mrQ$: it is the distance from the root $x_0$ to the nearest vertex of $\mcQ$. The two edges of $\mcQ^{x_0}$ created by splitting $\me_0$ at the root have lengths $x_0$ and $L_{\me_0}-x_0$, which are not bounded from below by $m$.

Otherwise, if $\mrQ$ is finite, denote by $\mcH_{|\mrQ|}$ the subspace of $\mcH$ generated by the first $|\mE_\mrQ|$ canonical basis vectors, where $|\mE_\mrQ|$ is the number of edges in $\mrQ$. Then the restriction
\begin{equation}
    \Psi_{\mrQ}: \mcH_{|\mrQ|}\to\mcH_{\mrQ}
\end{equation}
is bijective, and its inverse can be defined - with a slight abuse of notation - as
\begin{equation}\label{eq:PsiQinversefinite}
    \Psi_{\mrQ}^{-1}:\mcH_{\mrQ}\to \mcH_{|\mrQ|}\subset \mcH.
\end{equation}
In this case we also have that $\Psi_{\mrQ}^{-1}$ is bounded with
\begin{equation}\label{eq:PsiQnormestimate2}
\|\Psi_{\mrQ}^{-1}\|\leq \frac{1}{\sqrt{m'}}.
\end{equation}

By assumptions \eqref{eq:QstarDmMK}, using Remark \ref{rem:contwKNcond} and Proposition \ref{prop:AQgen}, we have that the operator $(A_{\mrQ},\mcD(A_{\mrQ}))$ is bounded from above by $0$ and the semigroup $(S_{\mrQ}(t))_{t\geq 0}$ is contractive on $\mcH_{\mrQ}.$
 \vspace{0.5cm}

In the following, we will need to ``transform'' the semigroup $(S_{\mrQ}(t))_{t\geq 0}$ so that it acts on $\mcH$ instead of $\mcH_{\mrQ}$.

 \begin{defi}\label{defi:SQtonmcH}
 If $\mrQ$ is infinite, we define
 \begin{equation}
     S_{\mrQ,\mcH}(t)f\coloneqq \Psi_{\mrQ}^{-1}S_{\mrQ}(t)\Psi_{\mrQ} f=\Psi_{\mrQ}^{-1}S_{\mrQ}(t)f_{\mrQ} ,\quad t\geq 0, f\in\mcH.
 \end{equation}
 It is straightforward that $\left(S_{\mrQ,\mcH}(t)\right)_{t\geq 0}$ is a strongly continuous semigroup on $\mcH$.

 If $\mrQ$ is finite, using \eqref{eq:PsiQinversefinite} we define  \begin{equation}\label{eq:SQonmcH}
   (S_{\mrQ,\mcH}(t)f)_{j}\coloneqq\begin{cases}
   (\Psi_{\mrQ}^{-1}S_{\mrQ}(t)\Psi_{\mrQ} f)_{j},&  1\leq j\leq |\mE_{\mrQ}|, \\
   f_j, & j>|\mE_{\mrQ}|, \end{cases}\quad t\geq 0,\quad f\in\mcH.
\end{equation}
It is easy to see that $\left(S_{\mrQ,\mcH}(t)\right)_{t\geq 0}$  is  a strongly continuous semigroup on $\mcH$ in this case as well.
\end{defi}

We also define the generator of the above semigroup acting on $\mcH$.

\begin{defi}
Let
\begin{equation}\label{eq:domAQonmcH}
\mcD\coloneqq \{f\in\mcH^2:f_{\mrQ}=\Psi_{\mrQ} f\in\mcD(A_{\mrQ})\}.
\end{equation}
If $\mrQ$ is infinite let
\begin{equation}\label{eq:AQonmcHi}
   A_{\mrQ,\mcH}f\coloneqq \Psi_{\mrQ}^{-1}A_{\mrQ}\Psi_{\mrQ} f=\Psi_{\mrQ}^{-1}A_{\mrQ}f_{\mrQ},\quad f\in\mcD(A_{\mrQ,\mcH}),
\end{equation}
and $\mcD(A_{\mrQ,\mcH})=\mcD$.

If $\mrQ$ is finite let
\begin{equation}\label{eq:AQonmcHf}
   (A_{\mrQ,\mcH}f)_{j}\coloneqq\begin{cases}
   (\Psi_{\mrQ}^{-1}A_{\mrQ}\Psi_{\mrQ} f)_{j},&  1\leq j \leq |\mE_{\mrQ}|, \\
   0, & j>|\mE_{\mrQ}|, \end{cases}\quad  f\in\mcD.
\end{equation}

 It is straightforward that if $\mrQ$ is infinite, the operator $\left(A_{\mrQ,\mcH},\mcD\right)$ is the generator of the strongly continuous semigroup $\left(S_{\mrQ,\mcH}(t)\right)_{t\geq 0}$ on $\mcH$. If $\mrQ$ is finite, the generator of $\left(S_{\mrQ,\mcH}(t)\right)_{t\geq 0}$ -- denoted by $\left(A_{\mrQ,\mcH},\mcD(A_{\mrQ,\mcH})\right)$ -- is the closure of the operator defined on $\mcD$ by \eqref{eq:AQonmcHf}; in particular, $\mcD$ is a core for it. 
\end{defi}

 By \eqref{eq:PsiQnormestimate1}, \eqref{eq:PsiQnormestimate2} and the contractivity of $(S_{\mrQ}(t))_{t\geq 0}$ on $\mcH_{\mrQ}$ we also have that
\begin{equation}\label{eq:SQnormonmcH}
\|S_{\mrQ,\mcH}(t)\|\leq \sqrt{\frac{M}{m'}},\quad t\geq 0,
\end{equation}
with $m'=m'(\mrQ)$ from \eqref{eq:rootmargin}. Note that this bound depends on the rooted quantum graph through its root margin; a bound uniform over $\mbKs$ is not available, since $\rho(\mrQ)$ can be arbitrarily small. The next lemma shows, however, that along $d$-convergent sequences the root margin -- and with it the above norm bounds -- can be controlled uniformly by the limit graph.

Let $r\in\nat^+$ and let $\mrQ=(\mcQ,x_0)$ and $\wmrQ=(\wmcQ,\wx)$ be two (finite or infinite) rooted quantum graphs in $\mbKs$ such that
\begin{equation}
    d(\mrQ,\wmrQ)<\frac{1}{1+r}
\end{equation}
holds. Then by Definition \ref{defi:metricrooted} and Remark \ref{rem:distKN} there exists an isomorphism of the $r$-balls around the roots
\begin{equation}
    \phi:\wmG(\wx,r)\xrightarrow{\sim}\mG(x_0,r)\text{ with }\beta^{x_0}=\phi\circ\wbeta^{\wx},
\end{equation}
such that
\begin{equation}\label{eq:LedistLemma}
    \begin{split}
        \max_{\me\in\mE_{\wmrQ}(\wx,r)}\left|\wL^{\wx}_{\me}-L^{x_0}_{\phi(\me)}\right|&<\frac{1}{r},\\
\max_{\me\in\mE_{\wmrQ}(\wx,r)}\sup_{t\in [0,1]}\left|W^{\wx}_{\me}(t)-W^{x_0}_{\phi(\me)}(t)\right|&<\frac{1}{r}.
    \end{split}
\end{equation}

\begin{lem}\label{lem:rootmargin}
Let $\mrQ=(\mcQ,x_0)$ and $\wmrQ=(\wmcQ,\wx)$ be rooted quantum graphs in $\mbKs$ with
\begin{equation}
    d(\mrQ,\wmrQ)<\frac{1}{1+r}\quad\text{for some integer } r\geq \max\left\{1,\,2/\rho(\mrQ)\right\},
\end{equation}
and let $\phi$ be the isomorphism of the $r$-balls from \eqref{eq:LedistLemma}. Then the following hold.
\begin{enumerate}
\item For every $k\leq r$, $\phi$ maps the set of vertices of $\wmG^{\wx}$ at distance exactly $k$ from $\mv_{\wx}$ bijectively onto the set of vertices of $\mG^{x_0}$ at distance exactly $k$ from $\mv_{x_0}$. In particular, the greatest distance from the root in $\mrQ$ is at least $r$ if and only if the same holds in $\wmrQ$; and if $\mrQ$ is finite with greatest root-distance less than $r$, then so is $\wmrQ$, and $\phi$ is an isomorphism of the entire rooted graphs.
\item Every edge of $\wmcQ^{\wx}$ has length at least $m''(\mrQ) \coloneqq\min\left\{\rho(\mrQ)/2,\,m\right\}$; consequently,
\begin{equation}\label{eq:transferbounds}
   \rho(\wmrQ)\geq\frac{\rho(\mrQ)}{2},\qquad \|\Psi_{\wmrQ}^{-1}\|\leq\frac{1}{\sqrt{m''(\mrQ)}},\qquad
    \|S_{\wmrQ,\mcH}(t)\|\leq\sqrt{\frac{M}{m''(\mrQ)}},\quad t\geq 0.
\end{equation}
\end{enumerate}
\end{lem}
\begin{proof}
(1) A shortest path from the root to a vertex at distance $k\leq r$ stays inside the $r$-ball, 
so distances up to $r$ are intrinsic to the $r$-balls, and both $\phi$ and $\phi^{-1}$ preserve adjacency and the root; hence $\phi$ preserves the distance-$k$ vertex sets for all $k\leq r$. If the distance-$r$ set of $\mrQ$ is empty, then so is that of $\wmrQ$; since the graphs are connected, every vertex then has distance at most $r-1$ from the root, so both graphs coincide with their finite $(r-1)$-balls and $\phi$ is an isomorphism of the entire rooted graphs.

(2) The two edges of $\wmcQ^{\wx}$ incident to the root vertex $\mv_{\wx}$ belong to the $1$-ball, and by \eqref{eq:LedistLemma} their lengths differ from $x_0$ and $L_{\me_0}-x_0$ by less than $1/r\leq\rho(\mrQ)/2$, so they are at least $\rho(\mrQ)/2$; all other edges of $\wmcQ^{\wx}$ are edges of $\wmcQ$ and have length at least $m$ by \eqref{eq:QstarDmMK}. The norm bounds then follow as in \eqref{eq:PsiQnormestimate2} and \eqref{eq:SQnormonmcH}.
\end{proof}

The next claim states that for such quantum graphs, for any $z\in\mcD(A_{\mrQ,\mcH})$, there exists $\wz\in\mcD(A_{\wmrQ,\mcH})$ such that the squared $\mcH^2$-norm of the difference $z-\wz$ can be estimated by $1/r^2$ times the squared $\mcH^2$-norm of $z$, plus the norm squares of the coordinates $z_j$ corresponding to edges not belonging to $\mE_\mrQ(x_0,r-1)$; see Figure~\ref{fig:rball} for an illustration of the balls involved.

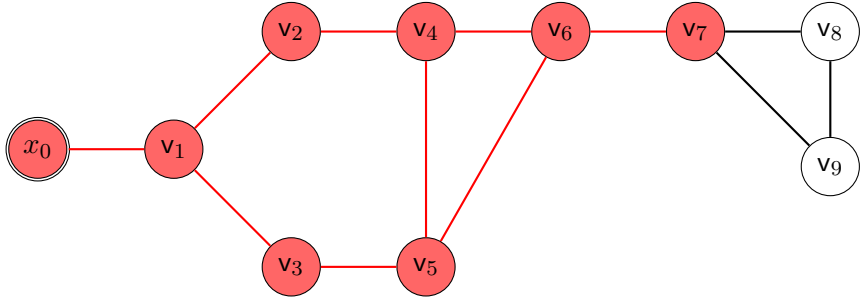
\begin{figure}[!ht]
\centering
\begin{tikzpicture}[
    vertex/.style={circle, draw, minimum size=7mm},
    redvertex/.style={vertex, fill=red!60},
    edge/.style={draw, thick},
    rededge/.style={edge, red},
    root/.style={redvertex, double}
]

\node[root] (x0) {$x_0$};                       
\node[redvertex, right=of x0] (v1) {$\mv_1$};     

\node[redvertex, above right=of v1] (v2) {$\mv_2$}; 
\node[redvertex, below right=of v1] (v3) {$\mv_3$}; 
\node[redvertex, right=of v2] (v4) {$\mv_4$};       
\node[redvertex, right=of v3] (v5) {$\mv_5$};       

\node[redvertex, right=of v4] (v6) {$\mv_6$};       
\node[redvertex, right=of v6] (v7) {$\mv_7$};       
\node[vertex, right=of v7] (v8) {$\mv_8$};          
\node[vertex, below=of v8] (v9) {$\mv_9$};          

\draw[rededge] (x0) -- (v1);
\draw[rededge] (v1) -- (v2);
\draw[rededge] (v1) -- (v3);
\draw[rededge] (v2) -- (v4);
\draw[rededge] (v3) -- (v5);
\draw[rededge] (v4) -- (v5);   
\draw[rededge] (v4) -- (v6);
\draw[rededge] (v6) -- (v7);
\draw[rededge] (v5) -- (v6);

\draw[edge] (v7) -- (v8);
\draw[edge] (v7) -- (v9);
\draw[edge] (v8) -- (v9);
\end{tikzpicture}
\caption{An example of $\mG(x_0,6)$. The vertices and the edges in $\mG(x_0,5)$ are colored in red; the ``boundary'' vertices are drawn as unfilled circles and the boundary edges in black.}
\label{fig:rball}
\end{figure}

\begin{prop}\label{lem:DAQDAQprimeonball}
Let $D\in\nat$, $0<m\leq M$ be given. Assume that $\mrQ=(\mcQ,x_0)$ and $\wmrQ=(\wmcQ,\wx)$ are two rooted quantum graphs in $\mbKs$ such that
\begin{equation}
    d(\mrQ,\wmrQ)<\frac{1}{1+r}\quad\text{for some integer }r\geq\max\left\{1,\,\frac{2}{\rho(\mrQ)}\right\}
\end{equation}
holds, with the root margin $\rho(\mrQ)$ from \eqref{eq:rootmargin}. Let $z\in\mcD\subset \mcD(A_{\mrQ,\mcH})$ be given, with $\mcD$ defined in \eqref{eq:domAQonmcH}.
\begin{enumerate}
    \item  If the greatest distance from the root in $\mrQ$ is at least $r$ (in particular, when $\mrQ$ is infinite) then there exists $K>0$, depending only on $D$, $m$, $M$ and $\rho(\mrQ)$, and $\wz\in\mcD(A_{\wmrQ,\mcH})$ such that
\begin{equation}\label{eq:lemmaclaim}
    \|z-\wz\|^2_{\mcH^2}\leq K\cdot \left(\frac{1}{r^2}\cdot  \|z\|^2_{\mcH^2}+\sum_{j\geq s_{r-1}}\|z_j\|^2_{H^2(0,1)}\right),
\end{equation}
where
\begin{equation}\label{eq:srminus1defi}
    s_{r-1}=|\mE_\mrQ(x_0,r-1)|=|\mE_{\wmrQ}(\wx,r-1)|.
\end{equation}
\item If $\mrQ$ is finite and the greatest distance from the root in $\mrQ$  is less than $r$ then there exists $\hat{K}>0$, depending only on $D$, $m$, $M$ and $\rho(\mrQ)$, and $\wz\in\mcD(A_{\wmrQ,\mcH})$ such that
\begin{equation}\label{eq:lemmaclaimfinite}
    \|z-\wz\|_{\mcH^2}\leq \hat{K}\cdot \frac{1}{r}\cdot  \|z\|_{\mcH^2}.
\end{equation}
\end{enumerate}

\end{prop}
\begin{proof}
Based on the above consideration, the assumption implies that there exists an isomorphism of the $r$-balls around the roots 
\begin{equation}
    \phi:\wmG(\wx,r)\xrightarrow{\sim}\mG(x_0,r)\text{ with }\beta^{x_0}=\phi\circ\wbeta^{\wx},
\end{equation} 
such that \eqref{eq:LedistLemma} holds. By Remark \ref{rem:psiQcompatible} we have
\begin{equation}\label{eq:zonQprimeball}
    z_{\wmrQ,\me}=z_{\mrQ,\phi(\me)},\quad \me\in\mE_{\wmrQ}(\wx,r).
\end{equation}

First we prove (1) and assume that the greatest distance from the root in $\mrQ$ is at least $r$. Let $z\in\mcD$ be given. Then $\Psi_{\mrQ}z\in\mcD(A_{\mrQ})$ holds, hence, the functions $\left(z_{\mrQ,\me}\right)_{\me\in\mE_{\mrQ}}$ satisfy the continuity and weighted Kirchhoff-Neumann conditions \eqref{eq:wcontcond}, \eqref{eq:wKNcond} in all vertices of $\mrQ$; in particular, in the vertices of $\mB_{\mrQ}(x_{0},r-1)$, where all the incident edges belong to $\mE_{\mrQ}(x_0,r)$.

Let \begin{equation}\label{eq:qdefi}
    q =|\mB_{\mrQ}(x_{0},r-1)|=|\mB_{\wmrQ}(\wx,r-1)|
    \end{equation}
be the number of the vertices in the $(r-1)$-balls around the roots in $\mrQ$ which is the same as in $\wmrQ$.
We define the linear operator $\Phi:\mcD(\Phi)\to\real^{q}$ as
\begin{equation}
    \Phi u \coloneqq
    \left(
        \sum\limits_{\me\in \mE_{\mv}(\wmrQ)}\frac{1}{\wL^{\wx}_{\me}}u_{\wmrQ,\me}'(\mv)
    \right)_{\mv\in \mB_{\wmrQ}(\wx,r-1)},\quad u\in \mcD(\Phi)
\end{equation}
with
\begin{equation}
    \mcD(\Phi)=\{u\in\mcH:u_j\in H^2(0,1)\text{ for all }1\leq j\leq s_r\},
\end{equation}
where, by the numbering procedure preceding \eqref{eq:Qnumbering}, the coordinates $u_j$ with $1\leq j\leq s_r=|\mE_{\wmrQ}(\wx,r)|$ are exactly those assigned to the edges of the $r$-ball of $\wmrQ$, and $\mE_{\mv}(\wmrQ)$ denotes the set of edges incident to $\mv$ in $\mE_{\wmrQ}$. Note that  for $\mv\in \mB_{\wmrQ}(\wx,r-1)$, all the edges incident to $\mv$ belong to $\mE_{\wmrQ}(\wx,r)$. 
By assumption, $z\in\mcD(\Phi)$ holds, and using \eqref{eq:zonQprimeball} we have
\begin{equation}\label{eq:Phiz}
    \Phi z= \left(
        \sum\limits_{\me\in \mE_{\mv}(\wmrQ)}\frac{1}{\wL^{\wx}_{\me}}z_{\mrQ,\phi(\me)}'(\mv)
    \right)_{\mv\in \mB_{\wmrQ}(\wx,r-1)}.
\end{equation}

\cite[Prop.~A.1]{KS23EJP}, applied to the finite quantum graph corresponding to the $r$-ball $\wmG(\wx,r)$, implies that there exists a function $w$ defined on the $r$-ball $\wmG(\wx,r)$, that is,
\begin{equation}
    w=\left(w_{\me}\right)_{\me\in \mE_{\wmrQ}(\wx,r)}
\end{equation} 
such that $w_{\me}\in H^2(0,1)$, $\me\in \mE_{\wmrQ}(\wx,r)$, the functions are continuous on $\wmG(\wx,r)$, and
\begin{equation}\label{eq:PhirwPhirz}
\begin{split}
    \sum\limits_{\me\in \mE_{\mv}(\wmrQ)}\frac{1}{\wL^{\wx}_{\me}}w_{\me}'(\mv)&=\Phi z,\quad \mv\in \mB_{\wmrQ}(\wx,r-1),\\
    \sum\limits_{\me\in \mE_{\mv}(\wmrQ)\cap \mE_{\wmrQ}(\wx,r)}\frac{1}{\wL^{\wx}_{\me}}w_{\me}(\mv)&=0,\quad \mv\in \mB_{\wmrQ}(\wx,r)\setminus \mB_{\wmrQ}(\wx,r-1).
    \end{split}
\end{equation}

The proof of \cite[Prop.~A.1]{KS23EJP} implies that for any $\me\in \mE_{\wmrQ}(\wx,r)$, $w_{\me}$ has the form
\begin{equation}\label{eq:wNform}
w_{\me}(x)=
	\alpha_{\me}\cdot e^{-\gamma x}+\beta_{\me}\cdot e^{-\gamma(1- x)},\quad x\in [0,1],
\end{equation}
for suitable constants $\alpha_{\me}$ and $\beta_{\me}
$ and a sufficiently large $\gamma>0$. It is straightforward from the proof that we can choose $\gamma$ to be large enough depending on the bounds for the edge lengths of the $r$-ball $\wmG(\wx,r)$; by Lemma \ref{lem:rootmargin}(2), these lie in $[m''(\mrQ),M]$ with $m''(\mrQ)=\min\{\rho(\mrQ)/2,\,m\}$, so $\gamma$ can be chosen depending only on $m$, $M$ and $\rho(\mrQ)$.

Then for a suitable constant $K(m,M,\rho(\mrQ))>0$, depending only on $m$, $M$ and $\rho(\mrQ)$,
we have
\begin{equation}\label{eq:wrnormestimate}
 \sum_{\me\in\mE_{\wmrQ}(\wx,r)}\|w_{\me}\|^2_{H^2(0,1)}\leq K(m,M,\rho(\mrQ))\cdot \|\Phi z\|^2_{\real^{q}}
\end{equation}
holds. \\

Notice that, since $\Psi_{\mrQ}z=z_{\mrQ}\in\mcD(A_{\mrQ})$, we have
\begin{equation}
   \sum_{\me\in \mE_{\mv}(\mrQ)}\frac{1}{L^{x_0}_{\me}}z_{\mrQ,\me}'(\mv)=0,\quad \text{ for all }\mv\in \mB_{\mrQ}(x_{0},r-1).
\end{equation}
Using this, \eqref{eq:QstarDmMK}, \eqref{eq:LedistLemma}, \eqref{eq:Phiz}, Lemma \ref{lem:rootmargin}(2) and Sobolev embedding, we obtain that there exists $k=k(D) > 0$ depending on $D$ (and independent of $r$) such that
\begin{equation}\label{eq:PhiNrzkicsi}
    \|\Phi z\|^2_{\real^{q}}= \sum_{\mv\in \mB_{\wmrQ}(\wx,r-1)}\left|\sum_{\me\in \mE_{\mv}(\wmrQ)}\left(\frac{1}{\wL^{\wx}_{\me}}-\frac{1}{L^{x_0}_{\phi(\me)}}\right)z_{\mrQ,\phi(\me)}'(\mv)\right|^2\leq \frac{k}{r^2}\cdot \frac{1}{m''(\mrQ)^4}\cdot \|z\|^2_{\mcH^2},
\end{equation}
where $z_{\mrQ,\phi(\me)}'$ is evaluated at the endpoint of $\phi(\me)$ corresponding to $\mv$ under the isomorphism $\phi$,
thus, by \eqref{eq:wrnormestimate}, there exists a constant $K'(D,m,M,\rho(\mrQ))$ (depending only on $D$, $m$, $M$ and $\rho(\mrQ)$) such that
\begin{equation}\label{eq:wrnormestimatebyKgamma}
  \sum_{\me\in\mE_{\wmrQ}(\wx,r)}\|w_{\me}\|^2_{H^2(0,1)}\leq \frac{K'(D,m,M,\rho(\mrQ))}{r^2}\cdot \|z\|^2_{\mcH^2}.
\end{equation}
\\

For the definition of $\wz$ we are going to need an interpolation on those edges of the $r$-ball in $\wmrQ$ which do not belong to the $(r-1)$-ball; such edges exist, since by Lemma \ref{lem:rootmargin}(1) the greatest distance from the root in $\wmrQ$ is at least $r$ as well.
Let $\me\in \mE_{\wmrQ}(\wx,r)\setminus \mE_{\wmrQ}(\wx,r-1)$ be given. First assume that one of the endpoints of $\me$ belongs to $\mB_{\wmrQ}(\wx,r-1)$; we may assume that $o(\me)\in \mB_{\wmrQ}(\wx,r-1)$ (here we make use of having actually oriented bonds, cf.~Definition \ref{defi:quantumgraph}). Using \eqref{eq:fQhatfQ} from Remark \ref{rem:psiQcompatible}, we define
\begin{equation}
    u_{\me}(x)\coloneqq \begin{cases}
        z_{\wmrQ,\me}(x)-w_{\me}(x), &\text{ if }x\in[0,\frac{1}{2}],\\
        p_{\me}(x), &\text{ if }x\in[\frac{1}{2},1],
    \end{cases}
\end{equation}
where $p_{\me}(x)$ is the third-degree Hermite interpolation polynomial with the data points
\begin{equation}\label{eq:Hermitedatapre}
\begin{split}
   p_{\me}\left(\frac{1}{2}\right)&=(z_{\wmrQ,\me}-w_{\me})\left(\frac{1}{2}\right), \quad p_{\me}'\left(\frac{1}{2}\right)=(z_{\wmrQ,\me}-w_{\me})'\left(\frac{1}{2}\right),\\
   p_{\me}(1)&=0,\quad 
   p'_{\me}(1)=0,
\end{split}
\end{equation}
that is
\begin{equation}\label{eq:Hermiteinterpolpolzminusw}
    p_{\me}(x)=p_1(x)\cdot (z_{\wmrQ,\me}-w_{\me})\left(\frac{1}{2}\right)+p_2(x)\cdot (z_{\wmrQ,\me}-w_{\me})'\left(\frac{1}{2}\right)
\end{equation}
with given basis polynomials $p_1,p_2$ of third degree.

If both endpoints of $\me$ belong to $\mB_{\wmrQ}(\wx,r)\setminus \mB_{\wmrQ}(\wx,r-1)$, we define
\begin{equation}\label{eq:uezero}
    u_{\me}\equiv 0.
\end{equation}
\\

Instead of defining $\wz\in\mcH$ directly, we first define $\Psi_{\wmrQ}\wz=\wz_{\wmrQ}\in\mcD(A_{\wmrQ})$ in the following way. Let
\begin{equation}\label{eq:zNrdefi}
    \wz_{\wmrQ,\me}\coloneqq \begin{cases}
        z_{\wmrQ,\me}-w_{\me}, & \me\in\mE_{\wmrQ}(\wx,r-1),\\
        u_{\me}, & \me\in \mE_{\wmrQ}(\wx,r)\setminus \mE_{\wmrQ}(\wx,r-1),\\
        0,& \me\notin \mE_{\wmrQ}(\wx,r).
    \end{cases}
\end{equation}
By a similar argument as in the proof of Theorem \ref{thm:heatsgrconv}, we have that $\wz_{\wmrQ,\me}\in H^2(0,1)$ holds for each $\me\in\mE_{\wmrQ}(\wx,r)$; since only finitely many coordinates of $\wz$ are nonzero, this implies $\wz_{\wmrQ}\in\mcH^2_{\wmrQ}$.

By its construction, $\wz_{\wmrQ}$ satisfies the continuity condition \eqref{eq:wcontcond} in the vertices of $\mB_{\wmrQ}(\wx,r-1)$. From \eqref{eq:PhirwPhirz} follows that the weighted Kirchhoff-Neumann conditions \eqref{eq:wKNcond} for the quantum graph $\wmrQ$ are satisfied in these vertices as well. Let now $\mv\in \mB_{\wmrQ}(\wx,r)\setminus \mB_{\wmrQ}(\wx,r-1)$ be any vertex and $\me\in\mE_{\mv}(\wmrQ)\cap \mE_{\wmrQ}(\wx,r)$ an arbitrary edge of the $r$-ball incident to $\mv$; then $\me\in \mE_{\wmrQ}(\wx,r)\setminus \mE_{\wmrQ}(\wx,r-1)$ and $\wz_{\wmrQ,\me}=u_{\me}$ holds. If $u_{\me}=0$, cf.~\eqref{eq:uezero}, then trivially $\wz_{\wmrQ,\me}(\mv)=\wz_{\wmrQ,\me}'(\mv)=0$. Otherwise $\mv=t(\me)$ holds, and by \eqref{eq:Hermitedatapre} we have
\begin{equation}\label{eq:boundaryverticeszero}
    \wz_{\wmrQ,\me}(\mv)=u_{\me}(1)=p_{\me}(1)=0,\quad  \wz_{\wmrQ,\me}'(\mv)=u_{\me}'(1)=p_{\me}'(1)=0.
\end{equation}
On the edges incident to $\mv$ not belonging to $\mE_{\wmrQ}(\wx,r)$, as well as on the edges incident to vertices outside $\mB_{\wmrQ}(\wx,r)$, the corresponding coordinates of $\wz_{\wmrQ}$ vanish identically. Hence the continuity and weighted Kirchhoff-Neumann conditions are trivially satisfied in all vertices of $\wmrQ$ outside $\mB_{\wmrQ}(\wx,r-1)$, and thus $\wz_{\wmrQ}\in\mcD(A_{\wmrQ})$ holds.  \\

Next we estimate the difference of $z_{\wmrQ}$ and $\wz_{\wmrQ}$ on an arbitrary edge \[\me\in \mE_{\wmrQ}(\wx,r)\setminus \mE_{\wmrQ}(\wx,r-1).\]
If $u_{\me}=0$, cf.~\eqref{eq:uezero}, then
\[\|z_{\wmrQ,\me}-\wz_{\wmrQ,\me}\|^2_{H^2(0,1)}=\|z_{\mrQ,\phi(\me)}\|^2_{H^2(0,1)}.\]
 Assume again $o(\me)\in \mB_{\wmrQ}(\wx,r-1)$.
By \eqref{eq:Hermiteinterpolpolzminusw} and Sobolev embedding there exists a constant $C>0$, depending only on $p_1$ and $p_2$, and a constant $k>0$, such that
\begin{equation}
\begin{split}
   \|z_{\wmrQ,\me}-\wz_{\wmrQ,\me}\|^2_{H^2(0,1)}&\leq 2\cdot \left(\|w_{\me}\|^2_{H^2(0,1)}+ \|z_{\mrQ,\phi(\me)}\|^2_{H^2(0,1)}+\|p_{\me}\|^2_{H^2(0,1)}\right)\\
   &\leq 2\cdot C\cdot \left(\|w_{\me}\|^2_{H^2(0,1)}+ \|z_{\mrQ,\phi(\me)}\|^2_{H^2(0,1)}+ \| z_{\mrQ,\phi(\me)}-w_{\me}\|^2_{C^1[0,1]}\right)
   \\
   &\leq 2\cdot C\cdot k\cdot \left(\|w_{\me}\|^2_{H^2(0,1)}+ \|z_{\mrQ,\phi(\me)}\|^2_{H^2(0,1)}+ \| z_{\mrQ,\phi(\me)}-w_{\me}\|^2_{H^2(0,1)}\right).
\end{split}
\end{equation}
Thus, for $\me\in \mE_{\wmrQ}(\wx,r)\setminus \mE_{\wmrQ}(\wx,r-1)$ we have
\begin{equation}\label{eq:zminuszrnormonErminusErminus1}
   \|z_{\wmrQ,\me}-\wz_{\wmrQ,\me}\|^2_{H^2(0,1)}
   \leq 8\cdot C\cdot k\cdot  \left(\|w_{\me}\|^2_{H^2(0,1)}+\| z_{\mrQ,\phi(\me)}\|^2_{H^2(0,1)}\right).
\end{equation}
Let now
\begin{equation}\label{eq:zhatdefi}
    \wz\coloneqq \Psi_{\wmrQ}^{-1}\wz_{\wmrQ}\in\mcD(A_{\wmrQ,\mcH}).
\end{equation}
Using \eqref{eq:zNrdefi}, \eqref{eq:zminuszrnormonErminusErminus1}, and assuming $8\cdot C\cdot k\geq 1$, we have that there exists $c>0$ such that 
\begin{equation}
    \begin{split}
     \|z-\wz\|^2_{\mcH^2}&=\sum_{\me\in\mE_{\wmrQ}(\wx,r)}\|z_{\wmrQ,\me}-\wz_{\wmrQ,\me}\|^2_{H^2(0,1)}+\sum_{\me\notin\mE_{\wmrQ}(\wx,r)}\|z_{\wmrQ,\me}\|^2_{H^2(0,1)}+\sum_{j>|\mE_{\wmrQ}|}\|z_{j}\|^2_{H^2(0,1)}\\
     &\leq \sum_{\me\in\mE_{\wmrQ}(\wx,r-1)}\|w_{\me}\|^2_{H^2(0,1)}+\sum_{\me\in\mE_{\wmrQ}(\wx,r)\setminus \mE_{\wmrQ}(\wx,r-1)}\|z_{\wmrQ,\me}-\wz_{\wmrQ,\me}\|^2_{H^2(0,1)}\\
     &+\sum_{\me\notin\mE_{\wmrQ}(\wx,r)}\|z_{\wmrQ,\me}\|^2_{H^2(0,1)}+\sum_{j>|\mE_{\wmrQ}|}\|z_{j}\|^2_{H^2(0,1)}\\
     &\leq c\cdot \left(\sum_{\me\in\mE_{\wmrQ}(\wx,r)}\|w_{\me}\|^2_{H^2(0,1)}+\sum_{j\geq s_{r-1}}\|z_{j}\|^2_{H^2(0,1)}\right).
    \end{split}
\end{equation}
Using \eqref{eq:wrnormestimatebyKgamma} we obtain that for an appropriate positive constant $K$, depending only on $D$, $m$, $M$ and $\rho(\mrQ)$,
\begin{equation}
   \|z-\wz\|^2_{\mcH^2}\leq K\cdot \left(\frac{1}{r^2}\cdot  \|z\|^2_{\mcH^2}+\sum_{j\geq s_{r-1}}\|z_{j}\|^2_{H^2(0,1)}\right),
\end{equation}
which is exactly the claim \eqref{eq:lemmaclaim}. \\

Let now $\mrQ$ be finite with greatest distance from the root less than $r$ and $z\in\mcD$ be given. We can proceed as in the infinite case, except for that by Lemma \ref{lem:rootmargin}(1), $\wmrQ$ is also finite with greatest root-distance less than $r$, and the set $\mE_{\wmrQ}(\wx,r)\setminus \mE_{\wmrQ}(\wx,r-1)$ is empty. Thus we can omit the functions $u_{\me}$ in \eqref{eq:zNrdefi}, and define
\begin{equation}
    \wz_{\wmrQ,\me}\coloneqq \begin{cases}
        z_{\wmrQ,\me}-w_{\me}, & \me\in\mE_{\wmrQ}(\wx,r),\\
        0,& \me\notin \mE_{\wmrQ}(\wx,r).
    \end{cases}
\end{equation}
Instead of \eqref{eq:zhatdefi}, we define
\begin{equation}
    \wz\coloneqq\begin{cases}
        (\Psi_{\wmrQ}^{-1}\wz_{\wmrQ})_j,& 1\leq j\leq |\mE_{\wmrQ}|,\\
        z_j,& j>|\mE_{\wmrQ}|.
    \end{cases}
\end{equation}
Then $\wz\in\mcD(A_{\wmrQ,\mcH})$ holds and by \eqref{eq:wrnormestimatebyKgamma} there exists $\hat{K}>0$, depending only on $D$, $m$, $M$ and $\rho(\mrQ)$ such that
\begin{equation}
    \|z-\wz\|_{\mcH^2}\leq \hat{K}\cdot \frac{1}{r}\cdot  \|z\|_{\mcH^2},
\end{equation}
which is exactly \eqref{eq:lemmaclaimfinite}.
\end{proof}

\vspace{0.5cm}

For any pair of $f,g\in \mcH$ and $t>0$, define the function
\begin{equation}\label{eq:Ffgr}
F_{f,g,t}(\mrQ)\coloneqq \langle g_{\mrQ},S_{\mrQ}(t)f_{\mrQ}\rangle_{\mcH_{\mrQ}} =\langle \Psi_{\mrQ} g,S_{\mrQ}(t)\Psi_{\mrQ} f\rangle_{\mcH_{\mrQ}},
\end{equation}
where  $\langle \cdot ,\cdot \rangle_{\mcH_{\mrQ}}$ denotes the usual scalar product on the Hilbert space $\mcH_{\mrQ}$, cf.~\eqref{eq:mcHtilde}.

\begin{rem}
Note that, if $\mrQ =(\mcQ, x_0)$ and $\wmrQ=(\wmcQ, \wx)$ are equivalent, then the operators $A_{\mrQ}$ and $A_{\wmrQ}$, associated with $\mcQ^{x_0}$ and $\wmcQ^{\wx}$  are unitarily equivalent, hence the semigroups $(S_{\mrQ}(t))_{t\geq 0}$ and $(S_{\wmrQ}(t))_{t\geq 0}$  generated by them are unitarily equivalent as well. 
Thus, by \eqref{eq:fQhatfQ}, we obtain that the function $F_{f,g,t}$ can be considered as a function on the equivalence classes $[\mrQ]=[\mcQ,x_0]$. 
\end{rem}

We now claim the boundedness and continuity of the functions $F_{f,g,t}$ on $\mbKs$.

\begin{theo}\label{prop:Ffgtrcontbdd}
For any $f,g\in \mcH$ and $t>0$, let the function $\mrQ\mapsto F_{f,g,t}(\mrQ)$ be defined in \eqref{eq:Ffgr}. Then for any $D\in\nat$, $0<m\leq M$, $F_{f,g,t}$ is bounded and continuous on $\mbKs$. 
\end{theo}
\begin{proof}
Let $D\in\nat$, $0<m\leq M$, $f,g\in\mcH$ and $t>0$ be fixed throughout the proof. Then for any rooted quantum graph $\mrQ=(\mcQ,x_0)\in \mbKs$, by the contractivity of $S_{\mrQ}(t)$ on $\mcH_{\mrQ}$ and \eqref{eq:mcHmchQnormequiv} we have
\begin{equation}
\begin{split}    
\left|F_{f,g,t}(\mrQ)\right|&= \left|\langle g_{\mrQ},S_{\mrQ}(t)f_{\mrQ}\rangle_{\mcH_{\mrQ}}\right|\leq \|f_{\mrQ}\|_{\mcH_{\mrQ}}\cdot\|g_{\mrQ}\|_{\mcH_{\mrQ}} \\ &\leq  M\cdot \|f\|_{\mcH}\cdot  \|g\|_{\mcH}
\end{split}
\end{equation}
holds.
Hence, $F_{f,g,t}$ is bounded on $\mbKs$.

Let \[[\mrQ_{N}]=[\mcQ_N,x_N],\; N\in\nat,\quad [\mrQ]= [\mcQ,x_0]\] and $[\mrQ_N]\to [\mrQ]$ be a convergent sequence in $\mbKs$, with respect to the metric $d$. We have to verify that
\begin{equation}\label{eq:FfgNtoFfg}
\lim_{N\to\infty}F_{f,g,t}(\mrQ_N)=\lim_{N\to\infty}\langle g_{\mrQ_N},S_{\mrQ_N}(t)f_{\mrQ_N}\rangle_{\mcH_{\mrQ_N}}=\langle g_{\mrQ},S_{\mrQ}(t)f_{\mrQ}\rangle_{\mcH_\mrQ}=F_{f,g,t}(\mrQ).
\end{equation}
Let first $\mrQ$ be infinite.  We will show that for any subsequence of $\left([\mrQ_N]\right)$, there exists a subsequence of it such that \eqref{eq:FfgNtoFfg} along this subsequence holds.
 
So as not to over-complicate the notations, assume that $\left([\mrQ_N]\right)$ is a subsequence of the original one. We will show that there exists a subsequence $\left([\mrQ_{N_r}]\right)_{r\in\nat}$ such that

\begin{equation}\label{eq:FfgNrtoFfg}
\lim_{r\to\infty}\langle g_{\mrQ_{N_r}},S_{\mrQ_{N_r}}(t)f_{\mrQ_{N_r}}\rangle_{\mcH_{\mrQ_{N_r}}}=\langle g_{\mrQ},S_{\mrQ}(t)f_{\mrQ}\rangle_{\mcH_\mrQ}.
\end{equation}
By Definition \ref{defi:metricrooted}, for each natural number $r\in\nat$ there exists $N_r\in\nat$, such that
\begin{equation}\label{eq:dQNrQN}
    d(\mrQ_{N_r},\mrQ)<\frac{1}{1+r},
\end{equation}
hence, the $r$-ball of the root in $\mrQ_{N_r}$, denoted by $\mG(x_{N_r},r)$, and the $r$-ball of the root in $\mrQ$, denoted by $\mG(x_0,r)$, are isomorphic via an isomorphism
 \begin{equation}\label{eq:isometry}
     \phi=\phi_r: \mG_{\mrQ_{N_r}}(x_{N_r}, r )\xrightarrow{\sim} \mG_{\mrQ}(x_0, r)    
 \end{equation}
 and for each $\me\in\mE_{\mrQ_{N_r}}(x_{N_r},r)$, the edge set of the $r$-ball in $\mrQ_{N_r}$, we have
\begin{align}
    \left|L^{x_{N_r}}_{\me}-L^{x_0}_{\phi(\me)}\right|&<\frac{1}{r},\label{eq:Ledist}\\
\sup_{t\in [0,1]}\left|W^{x_{N_r}}_{\me}(t)-W^{x_0}_{\phi(\me)}(t)\right|&<\frac{1}{r}.
\end{align}
Clearly, the sequence $(N_r)_{r\in\nat}$ can be chosen as strictly monotone increasing.

In what follows we only consider indices
\begin{equation}\label{eq:rrhorestriction}
    r\geq \max\left\{1,\,\frac{2}{\rho(\mrQ)}\right\};
\end{equation}
discarding finitely many members of the sequence does not affect the limit \eqref{eq:FfgNrtoFfg}. By Lemma \ref{lem:rootmargin}(2) and \eqref{eq:SQnormonmcH} we then have the bound
\begin{equation}\label{eq:SQNrHuniformbound}
    \|S_{\mrQ_{N_r},\mcH}(t)\|\leq \sqrt{\frac{M}{m''}},\qquad
    \|S_{\mrQ,\mcH}(t)\|\leq \sqrt{\frac{M}{m''}},\qquad t\geq 0,
\end{equation}
with $m''=m''(\mrQ)=\min\{\rho(\mrQ)/2,\,m\}$, uniformly in $r$ satisfying \eqref{eq:rrhorestriction}.

Let $\ve>0$ be fixed. We will show that  if $r$ is big enough then
\begin{equation}\label{eq:FfgNrminusFfg}
    \left|\langle g_{\mrQ_{N_r}},S_{\mrQ_{N_r}}(t)f_{\mrQ_{N_r}}\rangle_{\mcH_{\mrQ_{N_r}}}-\langle g_{\mrQ},S_{\mrQ}(t)f_{\mrQ}\rangle_{\mcH_\mrQ}\right|<\ve,
\end{equation}
and thus \eqref{eq:FfgNrtoFfg} holds.

If $f=0$ or $g=0$, then \eqref{eq:FfgNrminusFfg} holds trivially; hence we may assume $f\neq 0$, $g\neq 0$. Let $k_{\ve}$ be such that
\begin{equation}\label{eq:fjgjlessthanve}
    \sum_{j\geq k_{\ve}}\|g_j\|^2_{L^2(0,1)}<\frac{\ve^2\cdot m''}{16 M^3\cdot \|f\|^2_{\mcH}}.
\end{equation}\\

Since $\mrQ$ is infinite, $s_{r}=|\mE_{\mrQ}(x_0,r)|\to\infty$ as $r\to\infty$. Hence, we can choose $r$ satisfying \eqref{eq:rrhorestriction} such that
\begin{equation}\label{eq:rchoice}
    \frac{1}{r}<\frac{\ve\sqrt{m''}}{4\sqrt{M}\|f\|_{\mcH}\|g\|_{\mcH}}\text{ and }s_r:=|\mE_{\mrQ}(x_0,r)|>k_{\ve}.
\end{equation}

Since by the existence of the isomorphism in \eqref{eq:isometry}  the number of the edges in $\mE_{\mrQ_{N_r}}(x_{N_r},r)$ is the same as in $\mE_{\mrQ}(x_0,r)$, by \eqref{eq:rchoice} we have
\[s_r=|\mE_{\mrQ}(x_0,r)|=|\mE_{\mrQ_{N_r}}(x_{N_r},r)|>k_{\ve}.\]

We estimate the left-hand-side in \eqref{eq:FfgNrminusFfg} in two terms as
\begin{equation}\label{eq:Ffgtwoterms}
\begin{split}
    &\left|\langle g_{\mrQ_{N_r}},S_{\mrQ_{N_r}}(t)f_{\mrQ_{N_r}}\rangle_{\mcH_{\mrQ_{N_r}}}-\langle g_{\mrQ},S_{\mrQ}(t)f_{\mrQ}\rangle_{\mcH_\mrQ}\right|\leq \\
    & \left|\sum_{\me \in \mE(x_{N_r},r)}\int_0^1\left((S_{\mrQ_{N_r}}(t)f_{\mrQ_{N_r}})_{\me}(x)g_{\mrQ_{N_r},\me}(x)L^{x_{N_r}}_{\me} -(S_{\mrQ}(t) f_{\mrQ})_{\phi(\me)}(x)\, g_{\mrQ,\phi(\me)}(x) \, L^{x_0}_{\phi(\me)}\right)\dx \right|  \\
    &+\left|\sum_{\me \notin \mE(x_{N_r},r)}\int_0^1(S_{\mrQ_{N_r}}(t)f_{\mrQ_{N_r}})_{\me}(x)g_{\mrQ_{N_r},\me}(x)L^{x_{N_r}}_{\me}\dx -\sum_{\me \notin \mE(x_{0},r)}\int_0^1(S_{\mrQ}(t) f_{\mrQ})_{\me}(x)\, g_{\mrQ,\me}(x) \, L^{x_0}_{\me}\dx \right|
\end{split}
\end{equation}

For the second term, by \eqref{eq:PsiQnormestimate1}, \eqref{eq:SQNrHuniformbound}, \eqref{eq:fjgjlessthanve} and \eqref{eq:rchoice} and passing to the semigroups acting on $\mcH$, see Definition \ref{defi:SQtonmcH},  we have
\begin{equation}\label{eq:SQNrSQoutsideball}
    \begin{split}
        &\left|\sum_{\me \notin \mE(x_{N_r},r)}\int_0^1(S_{\mrQ_{N_r}}(t)f_{\mrQ_{N_r}})_{\me}(x)g_{\mrQ_{N_r},\me}(x)L^{x_{N_r}}_{\me}\dx -\sum_{\me \notin \mE(x_{0},r)}\int_0^1(S_{\mrQ}(t) f_{\mrQ})_{\me}(x)\, g_{\mrQ,\me}(x) \, L^{x_0}_{\me}\dx \right|\\
        &\leq M\cdot\left(\sum_{j\geq s_r}\left\|(S_{\mrQ_{N_r},\mcH}(t)f)_j\right\|_{L^2(0,1)} \cdot \left\|g_{j}\right\|_{L^2(0,1)}+\sum_{j\geq s_r}\left\|(S_{\mrQ,\mcH}(t)f)_j\right\|_{L^2(0,1)}\cdot\left\|g_{j}\right\|_{L^2(0,1)} \right)\\
    &\leq  M\cdot\left(\left\|S_{\mrQ_{N_r},\mcH}(t)f\right\|_{\mcH} +\left\|S_{\mrQ,\mcH}(t)f\right\|_{\mcH} \right)\cdot\left(\sum_{j\geq s_r}\left\|g_{j}\right\|^2_{L^2(0,1)}\right)^{1/2}\\
    &\leq \frac{2M \sqrt{M}}{\sqrt{m''}}\|f\|_{\mcH}\cdot\left(\sum_{j\geq s_r}\left\|g_{j}\right\|^2_{L^2(0,1)}\right)^{1/2}<\frac{\ve}{2}.
    \end{split}
\end{equation}

In the following we are going to estimate the first term in \eqref{eq:Ffgtwoterms}. Using \eqref{eq:SQNrHuniformbound}, \eqref{eq:Ledist}, \eqref{eq:fQhatfQ} and \eqref{eq:rchoice}, we obtain
\begin{align}
&\left|\sum_{\me \in \mE(x_{N_r},r)}\int_0^1\left((S_{\mrQ_{N_r}}(t)f_{\mrQ_{N_r}})_{\me}(x)g_{\mrQ_{N_r},\me}(x)L^{x_{N_r}}_{\me} -(S_{\mrQ}(t) f_{\mrQ})_{\phi(\me)}(x)\, g_{\mrQ,\phi(\me)} (x)\, L^{x_0}_{\phi(\me)}\right)\dx \right|\notag\\
&\leq\left| \sum_{\me \in \mE(x_{N_r},r)} \int_0^1
(S_{\mrQ_{N_r}}(t)f_{\mrQ_{N_r}})_{\me}(x)g_{\mrQ_{N_r},\me}(x)\left(L^{x_{N_r}}_{\me} - L^{x_0}_{\phi(\me)} \right)\, dx\right|\\
&+ \left| \sum_{\me \in \mE(x_{N_r},r)} \int_0^1 \left((S_{\mrQ_{N_r}}(t)f_{\mrQ_{N_r}})_{\me}(x)-(S_{\mrQ}(t)f_{\mrQ})_{\phi(\me)}(x)\right) \, g_{\mrQ, \phi(e)}(x)\, L^{x_0}_{\phi(e)} \, dx \right|\\
&\leq \max_{\me \in \mE(x_{N_r},r)}\left|L^{x_{N_r}}_{\me}-L^{x_0}_{\phi(\me)}\right|\cdot  \|S_{\mrQ_{N_r},\mcH}(t)f\|_{\mcH}\cdot \|g\|_{\mcH}\\
&+ M\cdot 
\|S_{\mrQ_{N_r},\mcH}(t)f-S_{\mrQ,\mcH}(t)f\|_{\mcH}\cdot \|g\|_{\mcH}\\
&\leq\frac{1}{r}\cdot  \sqrt{\frac{M}{m''}}\cdot\|f\|_{\mcH}\cdot \|g\|_{\mcH}+ M\cdot \|S_{\mrQ_{N_r},\mcH}(t)f-S_{\mrQ,\mcH}(t)f\|_{\mcH}\cdot\|g\|_{\mcH}\label{eq:Ffgtestimate}\\
&<\frac{\ve}{4}+M\cdot\|S_{\mrQ_{N_r},\mcH}(t)f-S_{\mrQ,\mcH}(t)f\|_{\mcH}\cdot\|g\|_{\mcH}.
\end{align}
We are going to show that if $r>0$ big enough then for the second term
\begin{equation}
    M\cdot \|S_{\mrQ_{N_r},\mcH}(t)f-S_{\mrQ,\mcH}(t)f\|_{\mcH}\cdot\|g\|_{\mcH}<\frac{\ve}{4}
\end{equation}
holds, that is,
\begin{equation}\label{eq:SQNrH-SQrHsmall}
    \|S_{\mrQ_{N_r},\mcH}(t)f-S_{\mrQ,\mcH}(t)f\|_{\mcH}<\frac{\ve}{4M\|g\|_{\mcH}}.
\end{equation}
 To do so, we prove that for any $f\in\mcH$,
 \begin{equation}
    \lim_{r\to\infty}\left\|S_{\mrQ_{N_r},\mcH}(t)f-S_{\mrQ,\mcH}(t)f\right\|_{\mcH}= 0
 \end{equation} 
is true, hence, for $r>0$ big enough, \eqref{eq:SQNrH-SQrHsmall} holds.

We are going to use the First Trotter--Kato Theorem from \cite[Thm.~III.4.8]{EN00}, applied to the sequence of semigroups $\left(S_{\mrQ_{N_r},\mcH}(t)\right)_{t\geq 0}$ indexed by the $r$ satisfying \eqref{eq:rrhorestriction}. The required uniform norm bound is provided by \eqref{eq:SQNrHuniformbound}. Hence, according to \cite[Thm.~III.4.8]{EN00}, we only have to show that for any $z\in\mcD$ -- which is a core for the generator $A_{\mrQ,\mcH}$ -- there exists a sequence $(z_{N_r})$ such that
\begin{equation}   z_{N_r}\in\mcD(A_{\mrQ_{N_r},\mcH}),r\in\nat,\; \lim_{r\to\infty }z_{N_r}= z\text{ and }\lim_{r\to\infty }A_{\mrQ_{N_r},\mcH} z_{N_r}= A_{\mrQ,\mcH}z\text{ in }\mcH
\end{equation}
are satisfied. By \eqref{eq:opAQ} and \eqref{eq:QstarDmMK}, the above two convergences hold true if
\begin{equation}\label{eq:zNRtoz}
    \lim_{r\to\infty}\|z-z_{N_r}\|_{\mcH^2}= 0.
\end{equation}\\

Let $z\in\mcD$ and $r$ satisfying \eqref{eq:rrhorestriction} be given. Recall that by \eqref{eq:dQNrQN}, $d(\mrQ_{N_r},\mrQ)<\frac{1}{1+r}$ holds.
Since $\mrQ$ is assumed to be infinite and $r\geq 2/\rho(\mrQ)$, we may apply Proposition \ref{lem:DAQDAQprimeonball}(1) with $\wmrQ=\mrQ_{N_r}$ and $\phi=\phi_r$, and obtain that there exists
\begin{equation}
z_{N_r}\coloneqq\wz\in\mcD(A_{\mrQ_{N_r},\mcH})
\end{equation}
such that \eqref{eq:lemmaclaim} holds, hence, with a constant $K>0$ depending only on $D$, $m$, $M$ and $\rho(\mrQ)$,
\begin{equation}
        \|z-z_{N_r}\|_{\mcH^2}^2 \leq K\cdot \left(\frac{1}{r^2}\cdot  \|z\|_{\mcH^2}^2+ \sum_{j\geq s_{r-1}}\|z_j\|^2_{H^2(0,1)}\right).
\end{equation}
Since $z\in\mcH^2$ and, $\mrQ$ being infinite, $s_{r-1}=|\mE_{\mrQ}(x_0,r-1)|\to\infty$ as $r\to\infty$, the right-hand side tends to $0$ as $r\to\infty$. This implies the claim \eqref{eq:zNRtoz}.\\

If $\mrQ$ is finite, we proceed in an analogous way as in the infinite case. For any sequence $(\mrQ_N)$ from $\mbKs$ with $[\mrQ_N]\to [\mrQ]$ we choose a subsequence $(\mrQ_{N_r})$ such that
\begin{equation}
    d(\mrQ_{N_r},\mrQ)<\frac{1}{1+r},\quad r\in\nat
\end{equation}
is satisfied. As before, we only consider indices $r$ satisfying \eqref{eq:rrhorestriction}, so that the uniform norm bound \eqref{eq:SQNrHuniformbound} is true. Applying the triangle inequality as in \eqref{eq:Ffgtwoterms} and using the finiteness of $\mrQ$, we obtain that for $r$ big enough
\begin{equation}
\begin{split}
    &\left|\langle g_{\mrQ_{N_r}},S_{\mrQ_{N_r}}(t)f_{\mrQ_{N_r}}\rangle_{\mcH_{\mrQ_{N_r}}}-\langle g_{\mrQ},S_{\mrQ}(t)f_{\mrQ}\rangle_{\mcH_\mrQ}\right|\leq \\
    & \left|\sum_{\me \in \mE(x_{N_r},r)}\int_0^1\left((S_{\mrQ_{N_r}}(t)f_{\mrQ_{N_r}})_{\me}(x)g_{\mrQ_{N_r},\me}(x)L^{x_{N_r}}_{\me} -(S_{\mrQ}(t) f_{\mrQ})_{\phi(\me)}(x)\, g_{\mrQ,\phi(\me)}(x) \, L^{x_0}_{\phi(\me)}\right)\dx \right|  \\
    &+\left|\sum_{\me \notin \mE(x_{N_r},r)}\int_0^1(S_{\mrQ_{N_r}}(t)f_{\mrQ_{N_r}})_{\me}(x)g_{\mrQ_{N_r},\me}(x)L^{x_{N_r}}_{\me}\dx \right|.
\end{split}
\end{equation}
For the second term, note that if $r$ additionally exceeds the greatest root-distance in $\mrQ$, then by Lemma \ref{lem:rootmargin}(1) the isomorphism $\phi$ identifies the entire rooted graphs; in particular, $\mE(x_{N_r},r)$ exhausts the edge set of $\mrQ_{N_r}^{x_{N_r}}$, so the sum over $\me\notin\mE(x_{N_r},r)$ is void and the second term vanishes.

For the first term we can again use triangle inequality and obtain for the estimate in \eqref{eq:Ffgtestimate}
\begin{equation}
\begin{split}
     &\left|\sum_{\me \in \mE(x_{N_r},r)}\int_0^1\left((S_{\mrQ_{N_r}}(t)f_{\mrQ_{N_r}})_{\me}(x)g_{\mrQ_{N_r},\me}(x)L^{x_{N_r}}_{\me} -(S_{\mrQ}(t) f_{\mrQ})_{\phi(\me)}(x)\, g_{\mrQ,\phi(\me)}(x) \, L^{x_0}_{\phi(\me)}\right)\dx \right| \\
    &\leq\frac{1}{r}\cdot  \sqrt{\frac{M}{m''}}\cdot\|f\|_{\mcH}\cdot \|g\|_{\mcH}+ M\cdot \|S_{\mrQ_{N_r},\mcH}(t)f-S_{\mrQ,\mcH}(t)f\|_{\mcH}\cdot\|g\|_{\mcH}.
\end{split}
\end{equation}
 To prove that
\begin{equation}
   \lim_{r\to\infty}\|S_{\mrQ_{N_r},\mcH}(t)f-S_{\mrQ,\mcH}(t)f\|_{\mcH} =0
\end{equation}
holds, one can use again the First Trotter--Kato Theorem, with the uniform norm bound from \eqref{eq:SQNrHuniformbound}. Since $\mrQ$ is finite, for $r$ large enough -- namely, $r$ satisfying \eqref{eq:rrhorestriction} and exceeding the greatest root-distance in $\mrQ$ -- Proposition \ref{lem:DAQDAQprimeonball}(2) applies with $\wmrQ=\mrQ_{N_r}$. Hence, we obtain that for any $z\in\mcD$, which is a core for the generator $A_{\mrQ,\mcH}$, one can choose a function $z_{N_r}\in\mcD(A_{\mrQ_{N_r},\mcH})$ such that for some $K=K(D,m,M,\rho(\mrQ))$,
\begin{equation}
\|z-z_{N_r}\|_{\mcH^2}\leq \frac{K}{r}\cdot  \|z\|_{\mcH^2}. 
\end{equation}
\end{proof}

According to Remark \ref{rem:Qstarsepcomplete}, we can endow $\mbQs$ with the Borel $\sigma$-algebra generated by the open sets with respect to the metric $d$. Let $\mcP(\mbQs)$ be the set of probability measures on $\mbQs$. Using this, we are able to define the Benjamini--Schramm limit of quantum graphs, cf.~\cite[Def.~3.5]{AISW2021}

\begin{defi}
Any finite quantum graph $\mcQ =(\mG, W,\beta, U)$ defines a probability measure $\nu_{\mcQ}\in \mcP(\mbQs)$ obtained by choosing a root uniformly at random:
\begin{equation}\label{eq:Qmeasure} 
\nu_{\mcQ}\coloneqq\frac{1}{\mcL(\mcQ)}\int_{\mG}\delta_{[\mcQ,x]}\dx,
\end{equation}
where $\delta_{[\mcQ,x]}$ is the Dirac measure of the equivalence class $[\mcQ,x]$. (For the precise definition of the integral see \eqref{eq:integralbound}; in particular, in accordance with the remark after Definition \ref{defi:Qx0}, the integral averages over the two orientations of each edge.)
If $\left(\mcQ_N\right)$ is a sequence of finite quantum graphs, we say that $\bbP\in\mcP(\mbQs)$ is the \emph{local weak limit} of $\left(\mcQ_N\right)$, or that $\left(\mcQ_N\right)$ \emph{converges in the sense of Benjamini--Schramm} to $\bbP$, if $\left(\nu_{\mcQ_N}\right)$ converges weakly-* to $\bbP$, that is, for every bounded continuous function $F\colon \mbQs\to\real$ we have
\begin{equation}\label{eq:BSlimitbddcontfct}
\frac{1}{\mcL(\mcQ_N)}\int_{\mG_N}F\left([\mcQ_N,x]\right)\dx \to \int_{\mbQs}F\left([\mcQ,x]\right)\,\mathrm{d}\bbP([\mcQ,x]).
\end{equation}
\end{defi}

As in the above definition, for a rooted quantum graph $\mrQ=(\mcQ,x)$ it will be important to emphasize which point is its root. Hence, in what follows, we will write this out as $(\mcQ,x)$ instead of the shorter notations using $\mrQ$ introduced in Notation \ref{notat:mrQmcQx}.

\begin{cor}\label{cor:FBSsubseq}
Let $D\in\nat$, $0<m\leq M$ and let $(\mcQ_N)$ be a sequence of finite quantum graphs satisfying \eqref{eq:QstarDmMK} for all $N\in\nat$. Then there exists a subsequence $(\mcQ_{N_k})$ and $\bbP\in\mcP(\mbKs)$ such that for any $f,g\in\mcH$ and $t>0$,
\begin{equation}
\begin{split}
    \lim_{k\to\infty}&\frac{1}{\mcL(\mcQ_{N_k})}\int_{\mG_{N_k}}\langle g_{(\mcQ_{N_k},x)},S_{(\mcQ_{N_k},x)}(t)f_{(\mcQ_{N_k},x)}\rangle_{\mcH_{(\mcQ_{N_k},x)}} \dx\\
    &=\int_{\mbKs}  \langle g_{(\mcQ,x)},S_{(\mcQ,x)}(t)f_{(\mcQ,x)}\rangle_{\mcH_{(\mcQ,x)}}\,\mathrm{d}\bbP([\mcQ,x]).
\end{split}
\end{equation}
\end{cor}
\begin{proof}
By a straightforward modification of the proof of \cite[Lem.~3.6]{AISW2021}, combined with the analogous statement of \cite[Cor.~3.7]{AISW2021} for $\mbKs$, we obtain that there is a subsequence of $(\mcQ_{N_k})$ which is convergent in the sense of Benjamini--Schramm to a probability  measure $\bbP\in\mcP(\mbKs)$. Hence, Theorem \ref{prop:Ffgtrcontbdd} implies the claim.
\end{proof}

Combining this result and Theorem \ref{thm:heatsgrconv}, we obtain the following.

\begin{theo}\label{thm:QNdoublelimit}
    Let $D\in\nat$, $0<m\leq M$ and $(\mcQ_N)$ be a sequence of finite quantum graphs satisfying \eqref{eq:QstarDmMK}, and assume that $(\mcQ_N)$ converges  in the sense of Benjamini--Schramm to a probability measure $\bbP\in\mcP(\mbKs)$. Let $f,g\in\mcH$ and $t>0$ be given. Then
    \begin{align}
    \lim_{r\to\infty}\lim_{N\to\infty}\frac{1}{\mcL(\mcQ_{N})}&\int_{\mG_{N}}\langle g_{(\mcQ_{N},x)},J_{(\mcQ_N,x),r}S_{(\mcQ_N,x)_r}(t)P_{(\mcQ_N,x),r}f_{(\mcQ_{N},x)}\rangle_{\mcH_{(\mcQ_{N},x)}} \dx\notag\\
    =\lim_{N\to\infty}\lim_{r\to\infty}\frac{1}{\mcL(\mcQ_{N})}&\int_{\mG_{N}}\langle g_{(\mcQ_{N},x)},J_{(\mcQ_N,x),r}S_{(\mcQ_N,x)_r}(t)P_{(\mcQ_N,x),r}f_{(\mcQ_{N},x)}\rangle_{\mcH_{(\mcQ_{N},x)}} \dx\notag\\
    =&\int_{\mbKs}  \langle g_{(\mcQ,x)},S_{(\mcQ,x)}(t)f_{(\mcQ,x)}\rangle_{\mcH_{(\mcQ,x)}}\,\mathrm{d}\bbP([\mcQ,x])\label{eq:Kstarintegral}
\end{align}
where, for a given rooted quantum graph $(\mcQ,x)$ and $r\in\nat$, the rooted quantum graph $(\mcQ,x)_r$ is defined in \eqref{eq:Qrdef}, the operators $J_{(\mcQ,x),r}$ and $P_{(\mcQ,x),r}$ are from Definition \ref{defi:PrJr}.
\end{theo}
\begin{proof}
       To verify that the first limit exists and equals the integral \eqref{eq:Kstarintegral}, first we show that for any $k\in\nat$,  the function
      \begin{equation}
          F_{f,g,t,k}(\mrQ)\coloneqq \langle g_{\mrQ},J_{\mrQ,k}S_{(\mrQ)_k}(t)P_{\mrQ,k}f_{\mrQ}\rangle_{\mcH_{\mrQ}}
      \end{equation} 
is bounded and continuous on $\mbKs$. (Here we again use the simpler notation $\mrQ$ for $(\mcQ,x)$ with a given root $x$; note that, compared with Section \ref{sec:approxheatsgr}, the roles of $k$ and $r$ are interchanged here.) The boundedness follows immediately.

Proceeding as in the proof of Theorem \ref{prop:Ffgtrcontbdd}, for the sequence $(\mrQ_N)$, $\mrQ_N=(\mcQ_N,x_N)$, $N\in\nat$ satisfying \eqref{eq:QstarDmMK} we choose an index sequence $(N_r)_{r\in\nat}$ such that
\begin{equation}
    d(\mrQ_{N_r},\mrQ)<\frac{1}{1+r},
\end{equation}
hence, the $r$-balls are isomorphic via an isomorphism
 \begin{equation}
     \phi=\phi_r: \mG_{\mrQ_{N_r}}(x_{N_r}, r )\xrightarrow{\sim} \mG_{\mrQ}(x, r)    
 \end{equation}
 and for each $\me\in\mE_{\mrQ_{N_r}}(x_{N_r},r)$, the edge set of the $r$-ball in $\mrQ_{N_r}$, we have
\begin{align}
    \left|L^{x_{N_r}}_{\me}-L^{x}_{\phi(\me)}\right|&<\frac{1}{r},\;
\sup_{t\in [0,1]}\left|W^{x_{N_r}}_{\me}(t)-W^{x}_{\phi(\me)}(t)\right|<\frac{1}{r}.
\end{align}
If $r>k$, the above holds also for the corresponding $k$-balls, and for any $f\in\mcH$, we have
\begin{equation}
    P_{\mrQ_{N_r},k}f_{\mrQ_{N_r}}=P_{\mrQ,k}f_{\mrQ}=:f_k
\end{equation}
being a function defined on (any of the isomorphic) $k$-balls $\mG_{\mrQ_{N_r}}(x_{N_r},k)$  and $\mG_{\mrQ}(x, k)$. 
Hence, restricting again to indices $r\geq\max\{1,\,2/\rho(\mrQ)\}$ and using the definition of $J_{\mrQ,k}$ together with Lemma \ref{lem:rootmargin}(2), one obtains, analogously to \eqref{eq:Ffgtestimate},
\begin{equation}
\begin{split}
&\left|F_{f,g,t,k}(\mrQ_{N_r})-F_{f,g,t,k}(\mrQ)\right|=
    \left|\langle g_{\mrQ_{N_r}},J_{\mrQ_{N_r},k}S_{(\mrQ_{N_r})_k}(t)f_{k}\rangle_{\mcH_{\mrQ_{N_r}}}-\langle g_{\mrQ},J_{\mrQ,k}S_{\mrQ_k}(t)f_{k}\rangle_{\mcH_\mrQ}\right|\\
    =
    & \left|\sum_{\me \in \mE(x_{N_r},k)}\int_0^1\left((S_{(\mrQ_{N_r})_k}(t)f_{k})_{\me}(x)g_{k,\me}(x)L^{x_{N_r}}_{\me} -(S_{\mrQ_k}(t) f_{k})_{\phi(\me)}(x)\, g_{k,\phi(\me)}(x) \, L^{x}_{\phi(\me)}\right)\dx \right| \\
    &\leq\frac{1}{r}\cdot  \sqrt{\frac{M}{m''(\mrQ)}}\cdot\|f\|_{\mcH}\cdot \|g\|_{\mcH}+ M\cdot \|S_{(\mrQ_{N_r})_k}(t)f_k-S_{\mrQ_k}(t)f_k\|_{\mcH_k}\cdot\|g\|_{\mcH},
\end{split}
\end{equation}
with $\mcH_k=L^2(0,1)^{s_k}$, where $s_k$ is the number of the edges in $\mE_{\mrQ_{N_r}}(x_{N_r},k)$ which is the same as in $\mE_{\mrQ}(x,k)$.  The proof can be now finished as the proof of Theorem \ref{prop:Ffgtrcontbdd} for the case of finite $\mrQ$, using Proposition \ref{lem:DAQDAQprimeonball}(2). Here the proposition is applied to the truncated graphs $(\mrQ)_k$ and $(\mrQ_{N_r})_k$: these are rooted quantum graphs in $\mbKs$ with the same root margin $\rho(\mrQ)$, and the labelings of the boundary vertices in \eqref{eq:betavr} may be chosen $\phi$-compatibly, so that $d\big((\mrQ)_k,(\mrQ_{N_r})_k\big)<1/(1+r)$ holds.

 Thus, for any sequence of quantum graphs $(\mcQ_N)$ satisfying \eqref{eq:QstarDmMK}, which is convergent in the sense of Benjamini--Schramm to a probability measure $\bbP\in\mcP(\mbKs),$  and for any $r\in\nat$ we have
\begin{equation}
\begin{split}
\lim_{N\to\infty}&\frac{1}{\mcL(\mcQ_{N})}\int_{\mG_{N}}\langle g_{(\mcQ_{N},x)},J_{(\mcQ_N,x),r}S_{(\mcQ_N,x)_r}(t)P_{(\mcQ_N,x),r}f_{(\mcQ_{N},x)}\rangle_{\mcH_{(\mcQ_{N},x)}} \dx\\
&=\int_{\mbKs}  \langle g_{(\mcQ,x)},J_{(\mcQ,x),r}S_{(\mcQ,x)_r}(t)P_{(\mcQ,x),r}f_{(\mcQ,x)}\rangle_{\mcH_{(\mcQ,x)}}\,\mathrm{d}\bbP([\mcQ,x])
\end{split}
\end{equation}
holds. Applying now Theorem \ref{thm:heatsgrconv} together with Lebesgue's dominated convergence theorem, we obtain the first convergence to the integral \eqref{eq:Kstarintegral}.\vspace{0.5cm}

To verify the second convergence, observe that $\mcQ_N$ is finite for each $N\in\nat$, implying
\begin{equation}
\begin{split}
\lim_{r\to\infty}&\frac{1}{\mcL(\mcQ_{N})}\int_{\mG_N}  \langle g_{(\mcQ_N,x)},J_{(\mcQ_N,x),r}S_{(\mcQ_N,x)_r}(t)P_{(\mcQ_N,x),r}f_{(\mcQ_N,x)}\rangle_{\mcH_{(\mcQ_N,x)}}\,\mathrm{d}x\\
&=\frac{1}{\mcL(\mcQ_{N})}\int_{\mG_N}  \langle g_{(\mcQ_N,x)}, S_{(\mcQ_N,x)}(t)f_{(\mcQ_N,x)}\rangle_{\mcH_{(\mcQ_N,x)}}\,\mathrm{d}x,
\end{split}
\end{equation}
since the sequence becomes constant in $r$. Thus, by assumption on the Benjamini--Schramm convergence of $(\mcQ_N)$, the claim for the second convergence follows.
\end{proof}    

By \cite[Cor.~3.7]{AISW2021} and the above theorem, the next claim follows.

\begin{cor}\label{cor:QNkdoublelimit}
    Let $D\in\nat$, $0<m\leq M$ and $(\mcQ_N)$ be a sequence of finite quantum graphs satisfying \eqref{eq:QstarDmMK}. Let $f,g\in\mcH$ and $t>0$ be given. Then there exist a subsequence $(\mcQ_{N_k})$ and a probability measure $\bbP\in\mcP(\mbKs)$ such that $(\mcQ_{N_k})$ converges to $\bbP$ in the sense of Benjamini--Schramm and
    \begin{align}
    \lim_{r\to\infty}\lim_{k\to\infty}\frac{1}{\mcL(\mcQ_{N_k})}&\int_{\mG_{N_k}}\langle g_{(\mcQ_{N_k},x)},J_{(\mcQ_{N_k},x),r}S_{(\mcQ_{N_k},x)_r}(t)P_{(\mcQ_{N_k},x),r}f_{(\mcQ_{N_k},x)}\rangle_{\mcH_{(\mcQ_{N_k},x)}} \dx\notag\\
    =\lim_{k\to\infty}\lim_{r\to\infty}\frac{1}{\mcL(\mcQ_{N_k})}&\int_{\mG_{N_k}}\langle g_{(\mcQ_{N_k},x)},J_{(\mcQ_{N_k},x),r}S_{(\mcQ_{N_k},x)_r}(t)P_{(\mcQ_{N_k},x),r}f_{(\mcQ_{N_k},x)}\rangle_{\mcH_{(\mcQ_{N_k},x)}} \dx\notag\\
    =&\int_{\mbKs}  \langle g_{(\mcQ,x)},S_{(\mcQ,x)}(t)f_{(\mcQ,x)}\rangle_{\mcH_{(\mcQ,x)}}\,\mathrm{d}\bbP([\mcQ,x]).
\end{align}
\end{cor}

\section{Examples of Benjamini--Schramm convergence}\label{sec:examplesBSlimit}

We close by mentioning three standard families of examples to which our results apply. In all of them the quantum graphs carry unit edge lengths, potential $W\equiv 0$, and continuity and Kirchhoff vertex conditions, so that \eqref{eq:QstarDmMK} holds with $m=M=1$ and the appropriate degree bound $D$.

\begin{exs}
(a) \emph{Cycles.} For $N\geq 3$ let $\mcQ_N$ be the cycle on $N$ vertices, with the labeling in which every vertex labels its two outgoing bonds according to a fixed orientation of the cycle. For $N>2r+2$, the labeled $r$-ball around any root of $\mcQ_N$ coincides (including all data) with the corresponding $r$-ball of the doubly infinite path $\mcQ_{\bbz}$ carrying the analogous orientation labeling. Hence $\left(\nu_{\mcQ_N}\right)$ converges weakly-* to the probability measure obtained by rooting $\mcQ_{\bbz}$ uniformly at a point of one (any) of its edges, thus Theorem \ref{thm:QNdoublelimit} applies.

(b) \emph{Discrete tori.} Similarly, let $\mcQ_N$ be the $N\times N$ discrete torus with the direction labeling (four fixed labels for the bonds pointing in the four coordinate directions). For $N>2r+2$ the labeled $r$-balls agree with those of the planar lattice $\bbz^2$, so $\left(\nu_{\mcQ_N}\right)$ converges weakly-* to $\bbz^2$ (with the direction labeling) rooted uniformly at a point of an edge.

(c) \emph{Random regular graphs.} Let $\mcQ_N$ be a uniformly random $d$-regular graph on $N$ vertices ($d\geq 3$, with $dN$ even so that such graphs exist), with labelings chosen uniformly at random. The random $d$-regular graph converges locally in probability to the rooted infinite $d$-regular tree, see \cite[Thm.~2.17]{Hofstad2024} or \cite[Examp.~19.7]{LovaszBook}. Consequently, $\left(\nu_{\mcQ_N}\right)$ converges weakly-* in probability to the distribution of the infinite $d$-regular tree with independent uniform labelings, rooted uniformly at a point of an edge. In particular, for this sequence Theorem \ref{thm:QNdoublelimit} applies.
\end{exs}



\end{document}